\documentclass[a4paper, 12pt]{amsart}
\usepackage{amsmath, amsthm, amssymb}
\usepackage{graphicx,mathrsfs}
\usepackage{subcaption, float}
\usepackage{appendix}
\usepackage{enumitem}
\usepackage[a4paper, total={6in, 8in}]{geometry}
\usepackage{microtype}
\usepackage[bookmarks=true,bookmarksnumbered=true, colorlinks=true, pdfstartview=FitV, linkcolor=blue, citecolor=blue, urlcolor=blue]{hyperref}
\usepackage{xcolor}

\usepackage{comment}

\usepackage[textsize=tiny]{todonotes}

\newtheoremstyle{mystyle}
  {3pt}
  {3pt}
  {\itshape}
  {}
  {\bfseries}
  {}
  { }
  {\thmname{#1}\space\thmnumber{#2}\thmnote{ (#3)}}

\theoremstyle{mystyle}
\newtheorem{theorem}{Theorem}[section]
\newtheorem{proposition}[theorem]{Proposition}
\newtheorem{lemma}[theorem]{Lemma}

\newtheorem{example}[theorem]{Example}
\newtheorem{assumption}[theorem]{Assumption}

\newcommand{\eS}{\mathrm{esssup}}

\def\R{\mathbb{R}}
\def\P{\mathbb{P}}

\newcommand{\dd}{{\mathrm{d}}}

\title{The Martingale Problem for Geometric stable-like Processes}

\author[S.R. Iyer]{Sarvesh Ravichandran Iyer}
\address{
    Visiting Faculty\\
    Mathematics Department\\
    Ashoka University\\
    Sonipat, Haryana 131029\\
    India
}
\email{sarvesh.iyer@ashoka.edu.in/sarveshiyer@gmail.com}


\numberwithin{equation}{section}

\begin{document}

\begin{abstract}
We prove that the martingale problem is well posed for pure-jump Lévy-type operators of the form $$
(\mathcal Lf)(x) = \int_{\mathbb R^d \setminus \{0\}} \left(f(x+h)-f(x) - (\nabla f(x) \cdot h)1_{\|h\| < 1}\right)K(x,h)\dd h,
$$
where $K(x,\cdot)$ is a jump kernel of the form $K(x,h) \sim \frac{l(\|h\|)}{\|h\|^d}$ for each $x \in \mathbb R^d,\|h\|<1$, and $l$ is a positive function that is slowly varying at $0$, under suitable assumptions on $K$. This includes jump kernels such as those of $\alpha$-geometric stable processes, $\alpha \in (0,2]$. 
\end{abstract}

\maketitle

\section{Introduction}

Given an operator $\mathcal{A}$, the existence and uniqueness in law of a strong Markov process whose infinitesimal generator is $\mathcal{A}$ is an important question. In many situations, the laws of such processes can be characterised as solutions to martingale problems. Then, the question of uniqueness in law is in many instances equivalent to the well-posedness of the associated martingale problem.

A class of operators for which the well-posedness of the martingale problem has been studied is that of symmetric pure jump Lévy-type operators $\mathcal L : C_b^2(\mathbb R^d) \to L^{\infty}(\mathbb R^d)$. These are of the form
\begin{align}\label{gsl}
(\mathcal Lf)(x) &= \int_{\mathbb R^d \setminus \{0\}} \left(f(x+h)-f(x) - (\nabla f(x) \cdot h)1_{\|h\| < 1}\right)K(x,h)\dd h \nonumber\\
& = \frac 12\int_{\mathbb R^d \setminus \{0\}} \left(f(x+h)+f(x-h) - 2f(x)\right)K(x,h)\dd h,
\end{align} 
where $K : \R^d \times \R^d \setminus \{0\} \to \R_+$ satisfies the following properties : \begin{gather*}
\int_{\mathbb R^d \setminus \{0\}} (1 \wedge \|h\|^2) K(x,h)\dd h < \infty \quad \text{for all } x \in \mathbb R^d, \text{ and} \\
K(x,h) = K(x,-h) \quad \text{for all } x \in \mathbb R^d, h \in \R^d \setminus \{0\}.
\end{gather*}
For the rest of this article, we refer to $K$ as the \emph{kernel} corresponding to $\mathcal{L}$.

In \cite{BT}, the martingale problem was proved to be well-posed for operators of the form \begin{equation}\label{eq:bt}
K(x,h) = \frac{A(x,h)}{\|h\|^{d+\alpha}},
\end{equation}
for some $\alpha \in (0,2)$ and $A : \mathbb R^d \times \R^d \setminus \{0\} \to \mathbb R_+$ which satisfies the following conditions : there exists $c_1,c_2>0$ such that \begin{equation}\label{eq:bddness}
c_1 \leq A(x,h) \leq c_2\quad \text{for all } x \in \R^d, h \in \R^d \setminus\{0\},
\end{equation} 
and there exists $\eta > 0$ such that for every $y \in \R^d$ and every $b > 0$,
\begin{equation} \label{eq:cty}
\lim_{x \to y} \sup_{\|h\|\leq b} |A(x, h) - A(y, h)| (1 + (\log
(1/\|x\|) \vee 0))^{1+\eta} = 0.
\end{equation}
These operators are linked to the $\alpha$-stable processes as follows. The $\alpha$-stable process on $\mathbb R^d, d \geq 1$, is a pure jump L\'{e}vy process with jump kernel given by $h \mapsto \frac{C_{\alpha}}{\|h\|^{d+\alpha}}$ for all $h \in \mathbb R^d \setminus \{0\}$ and some constant $C_{\alpha}>0$. Thus, the kernel $K$ defined in \eqref{eq:bt} is a perturbation $A$ of the jump kernel of an $\alpha$-stable process, where $A$ is bounded (see \eqref{eq:bddness}) and satisfies a continuity condition (see \eqref{eq:cty}). The processes that solve the martingale problem for a kernel $K$ of the form \eqref{eq:bt} are known as stable-like processes in the literature (see \cite{CK1} for the terminology). 

In this paper, we consider the geometric stable process of index $\alpha$, $\alpha\in (0,2)$ on $\mathbb R^d, d \geq 1$, which is a subordinate Brownian motion with characteristic exponent given by $$
\phi_{\alpha}(\lambda) = \log(1+\|\lambda\|^{\alpha}).
$$ 

The geometric stable process is a pure-jump L\'{e}vy process, and possesses a jump kernel $j_{\alpha} : \mathbb R^d \setminus \{0\} \to \mathbb R_+$ which satisfies the following asymptotic inequality : there exist $0<c_1<c_2$ such that 
$$
\frac{c_1}{\|h\|^d} \leq j_{\alpha}(h) \leq \frac{c_2}{\|h\|^d}\quad \text{for all $0<\|h\| \leq 1$}.
$$ 

If we define the kernel $K(x,h) = j_{\alpha}(h)$ in \eqref{gsl}, the resulting operator does not satisfy the assumptions of \cite{BT}. To the best of our knowledge, such kernels have not been studied before in the context of the martingale problem for operators of the form $\mathcal L$ in \eqref{gsl}.

Our aim in this paper is to prove the well-posedness of the martingale problem for operators whose kernels $K(x,\cdot) = A(x,\cdot)J(\cdot)$ are suitable perturbations $A(x,h)$ of the jump kernel $J(h)$ of a geometric stable process of index $\alpha$ for each $x \in \mathbb R^d$, where $\alpha \in (0,2]$. That is, for some constants $C_1,C_2>0$, $$
\frac{C_1}{\|h\|^{d}}\leq J(h) \leq \frac{C_2}{\|h\|^{d}}
$$
for all $x\in \R^d, 0 < \|h\|<1$, and $A$ satisfies a boundedness assumption similar to \eqref{eq:bddness} and a continuity assumption similar to \eqref{eq:cty}. Our proof technique is a suitable modification of the proof of \cite[Theorem 2.2]{B} and covers a larger class of kernels $K$, as described in Assumption~\ref{ass}.

The paper is organized as follows. In Section~\ref{mr}, we define the martingale problem and state the main result of this paper, Theorem~\ref{thm}. In Section~\ref{mli}, we discuss the motivation behind the problem and provide an outline of the proof. In Section~\ref{pf}, we state some preliminary lemmas and key propositions, and prove the theorem. These propositions are proved subsequently in Sections~\ref{sec:cty}, \ref{pfresreg}, \ref{pfrpb} and \ref{sec:loc}. 

Throughout this paper, we fix a dimension $d \geq 1$. All constants which may change from line to line are denoted by $C$, while those whose values are important will have a subscript e.g. $C_1,C_2$. All integration in this article will be performed with respect to the Lebesgue measure on $\R^d$. For a set $A \subset \mathbb R^d$, $1_A$ denotes the indicator function of $A$.

\section{Main result} \label{mr}

In this section, we state the main result of this paper, Theorem~\ref{thm}. We first define the martingale problem. Then we state our assumptions, followed by the theorem. Throughout this section, let $\mathcal L$ be as in \eqref{gsl} with kernel $K$.

We begin by introducing some function spaces needed for defining the martingale problem. Let $$
L^{\infty}(\R^d) = \{f : \mathbb R^d \to \mathbb R,\ \eS_{\mathbb R^d} \|f\| < \infty\}
$$ 
be the space of all essentially bounded (with respect to the Lebesgue measure) functions on $\mathbb R^d$, equipped with the norm $\|f\|_{\infty} = \eS_{\mathbb R^d} |f|$. Let $C_b^k(\mathbb R^d)$ be the set of all bounded functions $f(x_1,\ldots,x_d) : \mathbb R^d \to \mathbb R$ which are $k$-times differentiable such that all partial derivatives up to the $k$th order are bounded. $C_b^k(\mathbb R^d)$ is equipped with the norm
\begin{equation}\label{norm}
\|f\|_{C_b^k(\mathbb R^d)} = \sum_{|\alpha| \leq k} \left\|\frac{\partial^{|\alpha|} f}{\partial \mathbf{x}^{\alpha}}\right\|_{\infty}.
\end{equation}
Let $H_f$ denote the Hessian of $f$ i.e. the matrix of second partial derivatives of $f$, and $\|H_f\|_{\infty}$ be its spectral norm.

Let $\Omega = D([0,\infty); \mathbb R^d)$ be the set of all c\'{a}dl\'{a}g functions $f : [0,\infty) \to \mathbb R^d$ (i.e. those which are right continuous and possess left limits at all points). The space $\Omega$ is endowed with the Skorokhod topology (see \cite[Chapter 6]{JS} for the definition). Define the canonical coordinate process $\{X_t\}_{t \geq 0}$ by $X_t(\omega) = \omega(t)$ for all $\omega\in \Omega, t \geq 0$. Let $\{\mathcal F_t\}_{t \geq 0}$ be the right continuous augmentation of the natural filtration of $\{X_t\}_{t \geq 0}$ and let $\mathcal F_{\infty} = \sigma(\cup_{n \geq 1} \mathcal F_n)$.

We are now ready to define the martingale problem. Given $x \in \mathbb R^d$, a probability measure $\mathbb P_x$ on $\Omega$ solves the \emph{martingale problem} for $\mathcal L$ started at $x$ if $\mathbb P_x(X_0 = x) = 1$ and the process $\{M^f_t\}_{t \geq 0}$ given by \begin{equation}\label{eq:mart}
M^f_t = f(X_t) - f(X_0) - \int_0^t \mathcal Lf(X_s)ds
\end{equation}
is an $\{\mathcal F_t\}_{t \geq 0}$-martingale for all $f \in C_b^2(\mathbb R^d)$. A collection of probability measures $\{\mathbb P_x\}_{x \in \mathbb R^d}$ is called a \emph{strong Markov family} of solutions to the martingale problem for $\mathcal L$ if $\mathbb P_x$ solves the martingale problem for $\mathcal L$ started at $x$ for every $x \in \mathbb R^d$, and the strong Markov property holds for the process $\{X_t\}_{t \geq 0}$. That is, if $\mathbb E_x$ denotes the expectation with respect to $\P_x$ for every $x \in \mathbb R^d$, then
$$
\mathbb E_x[Y \circ \theta_T | \mathcal F_T] = \mathbb E_{X_T}[Y] \quad \mathbb P_x \text{-a.s.},
$$
for every $x \in \R^d$, finite stopping time $T$ and bounded $\mathcal{F}_{\infty}$-measurable random variable $Y$. Here, \begin{equation}\label{eq:shift}
(\theta_T(\omega))(s) = \omega(s + T(\omega))
\end{equation} 
is the shift operator by the random variable $T$ on $\Omega$. 

We will now state our assumption on the kernel $K$. 

\begin{assumption}\label{ass}
There exist $A : \mathbb R^d \times \mathbb R^d \setminus \{0\} \to \mathbb R_+$ and $J : \mathbb R^d \setminus \{0\} \to \mathbb R_+$ such that 
\begin{equation}\label{eq:dec}
K(x,h) = A(x,h)J(h)
\end{equation} 
for all $x \in \R^d, h \in \R^d \setminus\{0\}$. Further,
\begin{enumerate}[label = (\alph*)]
\item We have $\int_{\mathbb R^d \setminus \{0\}} (1 \wedge \|h\|^2)J(h)\dd h = K_0 < \infty$.
\item For all $x \in \R^d$, $h \in \mathbb R^d \setminus \{0\}$, $K(x,h)=K(x,-h)$.
\item There exist $\kappa>1$ and $l:(0,1) \to \mathbb R_+$ which is slowly varying at $0$ and satisfies \begin{equation}\label{eq:int}
\int_{0}^1 \frac{l(s)}{s} \dd s = +\infty,
\end{equation}
such that $$\kappa^{-1} \frac{l(\|h\|)}{\|h\|^d} \leq J(h) \leq \kappa \frac{l(\|h\|)}{\|h\|^d} \quad \text{for all } 0 < \|h\| \leq 1.$$ 
\item There exist $c_1,c_2>0$ such that $c_1 \leq A(x,h) \leq c_2$ for all $x \in \mathbb R^d, h \in \mathbb R^d \setminus \{0\}$.
\item There exists $\psi : \mathbb R_+ \to \mathbb R_+$ such that 
\begin{enumerate}[label = (\roman*)]
\item $\psi(s) = 1$ for all $s \geq 1$.
\item For all $b,\epsilon >0$ there exists $\delta >0$ such that \begin{equation}\label{eq:modcty}
\|y-x\| < \delta \implies \sup_{\|h\|<b}|A(x,h) -A(y,h)|\psi(\|h\|) < \epsilon.
\end{equation}
\item We have \begin{equation}\label{eq:modint}
\int_{\mathbb R^d \setminus \{0\}} \frac{J(h)}{\psi(\|h\|)} \dd h = \mathcal{J} < \infty.
\end{equation}
\end{enumerate}
\end{enumerate}
\end{assumption}

We now state the main result of this paper.

\begin{theorem}\label{thm}
There exists a unique strong Markov family of solutions $\{\mathbb P_x\}_{x \in \mathbb R^d}$ to the martingale problem for any operator $\mathcal{L}$ of the form \eqref{gsl} whose kernel $K$ satisfies Assumption~\ref{ass}.
\end{theorem}

In the next section, we motivate the martingale problem and outline the proof of Theorem~\ref{thm}.

\section{Motivation, literature and overview of proof}\label{mli}

We shall divide this section into three parts. In Section~\ref{sec:his} we shall briefly review the literature on the martingale problem. In Section~\ref{sec:mp} we shall discuss the motivation for the main result. Finally, in Section~\ref{sec:mppf} we conclude with an overview of Theorem~\ref{thm}.

\subsection{Literature on the martingale problem}\label{sec:his}

Suppose that $\{X_t\}_{t \geq 0}$ is a Feller process taking values in $\R^d$ with infinitesimal generator $\mathcal{A}$ defined on a domain $D(\mathcal{A})$. Then, by \cite[Proposition 1.7, Chapter 4]{EK}, for every $f \in D(\mathcal{A})$ the process $$
M^f_t = f(X_t) - f(X_0) - \int_0^t \mathcal{A}f(X_s)\dd s
$$
is a martingale with respect to the natural filtration of $\{X_t\}_{t \geq 0}$. Therefore, corresponding to every Feller process $\{X_t\}_{t \geq 0}$ is a family of martingales $\{M^f_t\}_{f \in D(\mathcal{A})}$. The martingale problem asks whether the reverse correspondence holds : for an operator $\mathcal{A}$, does the set of martingales $\{M^f_t\}_{f \in D(\mathcal{A})}$ correspond to a unique process $\{X_t\}_{t \geq 0}$?

The martingale problem formulation was introduced by Stroock and Varadhan\cite{SV1} as a probabilistic framework to accommodate the theory of existence and uniqueness of solutions for a class of backward partial differential equations (see \cite{K} for more details on this connection). They proved that the martingale problem is well-posed for operators of the form \begin{equation}\label{sv}
(\mathcal{A}u)(x) = \frac 12 \sum_{i,j=1}^d a_{ij}(u(x),x) \frac{\partial^2}{\partial x_i \partial x_j} + \sum_{i=1
}^d b_{i}(u(x),x) \frac{\partial}{\partial x_i},
\end{equation}
where $a_{ij} : \mathbb R^d \to \mathbb R$ are continuous and bounded functions for all $1 \leq i,j \leq d$ such that $[a_{ij}(x)]_{1 \leq i,j \leq d}$ is positive definite for all $x \in \mathbb R^d$, and $b_i : \mathbb R^d \to \R$ are bounded and measurable for all $1 \leq i \leq d$ (see \cite[Theorem 6.2]{SV1}).

They use a weak-convergence argument to prove the existence of a solution to the martingale problem. Their technique for proving uniqueness of solutions, for $b_i \equiv 0$, involves an $L^{p}$ estimate for the Green operator associated to the Brownian motion on $\R^d$ that is stable under perturbation of the generator. To adapt this to the case $b_i \not \equiv 0$, they change the underlying measure suitably using a Girsanov transform. We refer the reader to \cite[Section 5]{SV1} for the precise details.

The result of Stroock and Varadhan extended the work of Krylov\cite{KRY} who considered the stochastic differential equation
$$
X_t = x+\int_0^t \sigma(X_t)\dd\xi_t + \int_0^t b(X_t) \dd t 
$$ 
and proved existence and uniqueness of a solution in two dimensions under the assumption that the matrix valued function $\sigma$ is uniformly elliptic while the function $b$ is bounded. However, only existence could be proved in dimensions three or above, and the formulation as a martingale problem addresses the question of uniqueness in this setting for a much larger class of stochastic differential equations. Other applications of the martingale problem include an extension of the Donsker invariance theorem for Markov chains (see \cite[Section 10]{SV1}) and weak uniqueness for differential equations (see \cite[Theorem 5.4, Chapter VI]{Ba}).

\subsection{Motivation behind Theorem~\ref{thm}}\label{sec:mp}

We will now motivate our choice of working with operators of the form \eqref{gsl}. A pure-jump L\'{e}vy process $\{Y_t\}_{t \geq 0}$ on $\mathbb R^d,d \geq 1$ is one that has a L\'{e}vy exponent of the form $$
\phi_{Y}(\lambda) = \int_{\mathbb R^d \setminus \{0\}} (e^{i (\lambda \cdot h)} - 1 - i (\lambda \cdot h)) \dd \Pi(h) \quad \text{for all } \lambda \in \R^d, 
$$
where $\Pi$ is a measure on $\mathbb R^d \setminus \{0\}$ such that $\int_{\R^d \setminus \{0\}} (1 \wedge \|h\|^2) \dd \Pi(h) < \infty$. We call $\Pi(h)$ the L\'{e}vy measure of $\{Y_t\}_{t \geq 0}$. If $\Pi$ is absolutely continuous with respect to the Lebesgue measure, then its density $j : \R^d \setminus \{0\} \to \mathbb R_+$ is called the jump kernel of $\{Y_t\}_{t \geq 0}$. The infinitesimal generator of $\{Y_t\}_{t \geq 0}$ is then given by $$
\mathcal L_Yf(x) = \int_{\R^d \setminus \{0\}} (f(x+h)-f(x)-1_{\|h\|<1} (\nabla f(x)\cdot h)) j(h) \dd h
$$
for all $x \in \mathbb R^d$, $f\in C_b^2(\mathbb R^d)$. Thus, operators of the form \eqref{gsl} are mixtures of the infinitesimal generators of pure-jump L\'{e}vy processes. $\{Y_t\}_{t \geq 0}$ is called symmetric if $j(h) = j(-h)$ for $h \in \R^d \setminus \{0\}$. Symmetry is commonly seen as an assumption in the martingale problem.

The first works on the well-posedness of the martingale problem for processes having a non-trivial jump part were those of Stroock\cite{STR} and Komatsu (\cite{KOM1},\cite{KOM2}). Of these, the assumptions on the L\'{e}vy measure in \cite{KOM2} are phrased in terms of the symbol of the generator as a pseudo-differential operator. Assumptions of this kind will not be imposed in this paper. For other articles which make assumptions on the symbol of the generator and prove well-posedness of the associated martingale problem, we refer the reader to the articles of Hoh(\cite{HOH}), Jacob(\cite[Chapter 4]{JAC}) and Kühn(\cite{FK}). 

On the other hand, \cite{KOM1} and \cite{STR} prove that the martingale problem is well-posed for operators of the form $\mathcal{A} + \mathcal{K}$, where $\mathcal{A}$ is as in \eqref{sv} and $\mathcal{K}$ is of the form \eqref{gsl}. The proofs focus on obtaining the same $L^p$ estimates as derived in \cite{SV1}. Another method of obtaining these estimates that uses more analytic tools is that of Oleinik and Radkevich\cite{OR}, which was also used in the works of Mikulevičius and Pragarauskas (\cite{MP2}, \cite{MP1}).

Bass\cite{B}, proved the well-posedness of the martingale problem for operators of the form \eqref{gsl} such as those with kernel $$K(x,h) = \frac{\xi_{\zeta(x)}}{\|h\|^{1+\zeta(x)}},$$ where $\alpha : \mathbb R \to (0,2)$ is strictly bounded away from $0$ and $2$, the function $\beta(s) = \max_{\|x-y\| \leq s} |\zeta(x) - \zeta(y)|$ satisfies $$\int_0^1 \frac{\beta(s)}{s} \dd s < \infty\text{ and }\lim_{s \to 0} \beta(s) |\ln(s)| = 0,$$ and $\xi_{\zeta(x)}$ is a normalising constant for each $x \in \mathbb R$. While the proof of the existence of a solution coincides with the weak convergence technique used in the proof of \cite[Theorem 4.2]{SV1}, the proof of uniqueness avoids having to deal with probabilistic estimates like those of the Green function and heat kernels. To explain the method, let $\{X_t\}_{t \geq 0}$ be a solution to the martingale problem started at $x_0 \in \mathbb R^d$ corresponding to $K$, and let $\{Y_t\}_{t \geq 0}$ be the pure jump symmetric L\'{e}vy process with jump kernel $j_{x_0} : \mathbb R^d \setminus \{0\} \to \mathbb R_+$ given by $j_{x_0}(h) = K(x_0,h)$. The key idea is to estimate the semigroup corresponding to $\{X_{t}\}_{t \geq 0}$ by the semigroup corresponding to $\{Y_t\}_{t \geq 0}$. A demonstration of this proof technique may also be found in Bass and Perkins\cite{BP}.

Recall the main result of Bass and Tang, \cite[Theorem 1.2]{BT} from the introduction of this paper. While they also used the semigroup comparison method of Bass\cite{B}, the key difference in their estimate lies in the usage of the $L^2(\mathbb R^d)$ norm instead of the $L^{\infty}(\R^d)$ norm for comparing semigroups. The usage of this norm requires additional estimates such as bounds on the Fourier transform of the heat kernel(see \cite[Corollary 2.8]{BT}). 

We note that the geometric stable processes (or its associated operator/jump kernel) do not satisfy any of the assumptions imposed in any of the papers above, insofar as uniqueness of the martingale problem is concerned. The reason is that bounds such as heat kernel bounds, semigroup ultracontractivity or Green function bounds are either not known or do not satisfy the assumptions imposed in previous papers, such as those of \cite{B} and \cite{BT}. This is the main motivation behind this paper.

We complete this section with an example which illustrates the usage of Theorem~\ref{thm}. Our example is motivated from \cite[Corollary 2.3 and (2.6)]{B}.

We require the below lemma before stating our example.
\begin{lemma}\label{lem:exam}
Let $J,\beta  : \R^d \setminus \{0\} \to \mathbb R_+$, $\zeta : \mathbb R^d \to \mathbb R_+$ and $\psi : \R_+ \to \R_+$ be functions which satisfy the following conditions :
\begin{enumerate}[label = \Roman*)]
\item $J$ satisfies Assumption~\ref{ass}(a),(c) with some choice of $\kappa$ and $l$.
\item $c_1 \leq \zeta(x)\leq c_2$ and $c_1 \leq \beta(h) \leq c_2$ for some $c_1,c_2>0$ and every $x\in \R^d,h \in \R^d \setminus \{0\}$.
\item $J(h) = J(-h)$ and $\beta(h) = \beta(-h)$ for every $h \in \R^d \setminus \{0\}$.
\item $\psi : \R_+ \to \R_+$ satisfies Assumption~\ref{ass}(e)(i),(iii) for $J$.
\item $|\ln(\beta(h))|\psi(\|h\|) < c_3$ for some constant $c_3>0$ and every $h \in \R^d \setminus \{0\}$, and $\zeta$ is uniformly continuous.
\end{enumerate}

Let $A(x,h) = \beta(h)^{\zeta(x)}$. Then, the kernel $K(x,h) = A(x,h)J(h)$ satisfies Assumption~\ref{ass} and there is a unique strong Markov family of solutions to the martingale problem for any operator of the form \eqref{gsl} with kernel $K$. 
\end{lemma}

\begin{proof}
Assumption~\ref{ass}(a),(c) are satisfied from I), while Assumption~\ref{ass}(d) follows from the definition of $A$ and II). Assumption~\ref{ass}(b) follows by III), while Assumption~\ref{ass}(e)(i),(iii) follow from IV).
 
In order to verify Assumption~\ref{ass}(e)(ii), let $b,\epsilon>0$ be arbitrary. Suppose that $\|h\|<b$ and $x,y \in \R^d$. We have \begin{align*}
|A(x,h) - A(y,h)|\psi(\|h\|) &= |\beta(h)^{\zeta(x)} - \beta(h)^{\zeta(y)}|\psi(\|h\|) \\ & \leq |\zeta(x)-\zeta(y)| \max\{\beta(h)^{\zeta(x)},\beta(h)^{\zeta(y)}\} |\ln(\beta(h))| \psi(\|h\|), 
\end{align*}
where we used the mean value inequality and the monotone nature of $z \to \beta(h)^{z}$ above. Note that $\beta(h)^{\zeta(x)} = A(x,h)$ and $\beta(h)^{\zeta(y)} = A(y,h)$ by definition. Combining Assumption~\ref{ass}(d) with  V), $$
|\zeta(x)-\zeta(y)| \max\{\beta(h)^{\zeta(x)},\beta(h)^{\zeta(y)}\} |\ln(\beta(h))| \psi(\|h\|) \leq C'|\zeta(x)-\zeta(y)|
$$
for some constant $C'>0$ independent of $x,y$ and $h$. Combining the above inequalities and taking the supremum over $\|h\|<b$, \begin{equation}\label{eq:unif}
\sup_{\|h\|<b} |A(x,h) - A(y,h)|\psi(\|h\|) \leq C'|\zeta(x) - \zeta(y)|
\end{equation}
for every $b>0$. 

Let $\delta>0$ be such that $\|x-y\| < \delta \implies |\zeta(x) - \zeta(y)| < \frac{\epsilon}{C'}$. Then, $$
\|x-y\| < \delta \implies |\zeta(x) - \zeta(y)| < \frac{\epsilon}{C'} \implies \sup_{\|h\|<b} |A(x,h) - A(y,h)|\psi(\|h\|) < \epsilon
$$ 
by \eqref{eq:unif}. This shows that Assumption~\ref{ass}(e)(ii), and hence Assumption~\ref{ass} is satisfied by $K(x,h)$. The result now follows by applying Theorem~\ref{thm}. 
\end{proof}
\begin{example}
Let $J_{\alpha}$ be the jump kernel corresponding to the geometric stable process of index $\alpha \in (0,2]$ (See Section~\ref{mr}), which satisfies $J_{\alpha}(h)  = J_{\alpha}(-h)$ and Assumption~\ref{ass}(c) with $l \equiv 1$. Let $\zeta : \R^d \to \R_+$ be any uniformly continuous function which is bounded and bounded away from $0$ (for example, $\zeta(x) = 1+e^{-\|x\|}$) and let $\beta(h) = e^{1 \wedge \|h\|^{\epsilon}}$ for some fixed $\epsilon>0$. Set $\psi(s) = \frac{1}{s^{\epsilon} \wedge 1}$ for any $\epsilon\in (0,1)$.

  We note that all conditions I),II),III) and V) of Lemma~\ref{lem:exam} are satisfied by the above definitions, and only need to verify condition IV). For this, note that $\psi(s) = 1$ for $s \geq 1$. Thus, it is sufficient to prove that \begin{equation}\label{eq:part12}
\int_{\R^d \setminus \{0\}} \frac{J_{\alpha}(h)}{\psi(\|h\|)}\dd h < \infty.
\end{equation}
Clearly $\int_{\|h\|\geq 1} J_{\alpha}(h) \dd h < K_0$ since $J_{\alpha}$ satisfies Assumption~\ref{ass}(a). Now, by Assumption~\ref{ass}(c) and a change of variable, 
\begin{equation}\label{eq:part2}
\int_{0<\|h\|<1} \|h\|^{\epsilon}J_{\alpha}(h) \dd h \leq \kappa \int_{0<\|h\|<1} \|h\|^{\epsilon-d} \dd h = C\int_{0}^1 r^{\epsilon-1} \dd r = \frac{C}{\epsilon} < \infty
\end{equation}
since $\epsilon>0$. Combining this with \eqref{eq:part12} shows that $\psi$ satisfies condition V) of Lemma~\ref{lem:exam}. Thus, the kernel 
$$
K(x,h) = e^{\zeta(x)(\|h^{\epsilon}\| \wedge 1)}J_{\alpha}(h)
$$
satisfies Lemma~\ref{lem:exam} for any $\epsilon>0$. By an application of the lemma, there is a unique strong Markov family of solutions to the martingale problem for the operator associated to it via \eqref{gsl}. 

\end{example}

\section{Overview of the proof of Theorem~\ref{thm}}\label{sec:mppf}

We will now provide an overview of the proof of Theorem~\ref{thm}, beginning with comments on Assumption~\ref{ass}. In keeping with the assumptions in \cite{BT}, we ensure that the kernel $K$ is of the form $K(x,h) = A(x,h)J(h)$, where $A$ satisfies the boundedness assumption Assumption~\ref{ass}(d) which is the analogue of \eqref{eq:bddness}, and the continuity assumption Assumption~\ref{ass}(e) which is the analogue of \eqref{eq:cty}. 

Recall the definition of a jump kernel and the formula of the infinitesimal generator of a L\'{e}vy process from the previous section. Assumptions~\ref{ass}(a),(b) and (c) ensure that we restrict our attention to perturbing the generators of symmetric L\'{e}vy processes that possess a jump kernel $J$ which is slowly varying at $0$. This class of processes has not been covered previously in the history of the martingale problem, hence we choose to focus on it.

We shall now discuss the proof of Theorem~\ref{thm}. Suppose that $\mathcal{L}$ is of the form \eqref{gsl} such that its kernel $K$ satisfies Assumption~\ref{ass}. Our proof of the existence of a solution requires a key proposition (Proposition~\ref{prop:cty}), namely that $\mathcal{L}f$ is uniformly continuous for each $f \in C_b^3(\mathbb R^d)$. The proof of the existence of the solution now follows by a suitable modification of the proof of \cite[Theorem 2.1]{B}, using a weak convergence argument. A selection argument then gives the existence of a strong Markov solution to the problem.

The proof of uniqueness in Theorem~\ref{thm} requires three key propositions. The first of these, Proposition~\ref{lem:resreg} requires the definition of the resolvent. Given a strong Markov process $\{Y_t\}_{t \geq 0}$, its resolvent is defined by \begin{equation}\label{res}
(R_{\lambda}^Yg)(x) = \mathbb E_x \int_{0}^{\infty} e^{-\lambda t}g(Y_t)dt,
\end{equation}
for every $g \in L^{\infty}(\mathbb R^d)$, $x\in \R^d$ and $\lambda\geq 0$. The first key proposition, Proposition~\ref{lem:resreg} states that for any $\lambda>0$, the resolvent $R_{\lambda}g$ of a solution of the martingale problem is continuous on $\R^d$ for any $g \in L^{\infty}(\R^d)$. The proof relies on exit time estimates and a result on the regularity of harmonic functions with respect to a solution of the martingale problem.

The second proposition, Proposition~\ref{rpb} is a resolvent perturbation bound. More precisely, let $\{X_t\}_{t \geq 0}$ be a solution to the martingale problem started at $x_0 \in \mathbb R^d$ corresponding to $K$, and let $\{Y_t\}_{t \geq 0}$ be the pure jump symmetric L\'{e}vy process with jump kernel $j_{x_0} : \mathbb R^d \setminus \{0\} \to \mathbb R_+$ given by $j_{x_0}(h) = K(x_0,h)$. Let $R_{\lambda}^{x_0}$ and $\mathcal M^{x_0}$ be the resolvent and infinitesimal generator of $\{Y_t\}_{t \geq 0}$ respectively. Under an additional assumption on $K$, Assumption~\ref{xias} which is stronger than Assumption~\ref{ass}(e), Proposition~\ref{rpb} shows a bound of the form $$
\|(\mathcal L - \mathcal M^{x_0})(R_{\lambda}^{x_0}g)\|_{\infty} \leq C\|g\|_{\infty}
$$ 
where it can be ensured under some conditions that $0<C<1$. This proposition provides an estimate of how well $\mathcal M^{x_0}$ approximates $\mathcal L$.

Finally, the last proposition Proposition~\ref{prop:loc} is a localization argument, which relaxes Assumption~\ref{xias} to Assumption~\ref{ass}(e). It is a technical result that uses the structure of $D([0,\infty);\mathbb R^d)$. Informally, given a strong Markov solution to the martingale problem $\{X_t\}_{t \geq 0}$ with filtration $\{\mathcal F_{t}\}_{t \geq 0}$, define the stopping times  $\tau_i : \Omega \to \mathbb R_+$ for $i\geq 0$ by $$
\tau_0= 0, \tau_i = \inf\{t \geq \tau_{i-1} : \|X_{t} - X_{\tau_{i-1}}\| \geq \eta\} \text{ for } i \geq 1.
$$
The result shows that $\tau_i \to +\infty$ a.s., and then inductively shows that any solution to the martingale problem is uniquely determined on the stopped filtration $\mathcal F_i, i \geq 1$. The result follows after showing that $\mathcal F_{\infty}$ is generated by $\cup_{i=1}^{\infty} \mathcal F_i$.

Given these three propositions, the key idea of the proof is as follows. If $\{\mathbb P^i_x\}_{x \in \mathbb R^d}$ are two strong Markov families of solutions to the martingale problem, then we consider the resolvents under these measures $$
R^i_{\lambda} f(x) = \mathbb E^i_x \int_0^{\infty} e^{-\lambda t}f(X_t)dt, \quad i=1,2
$$
for $f\in L^{\infty}(\mathbb R^d)$ and $x\in \R^d$. Assuming Assumption~\ref{xias}, we show using the resolvent continuity and perturbation bound that for large enough $\lambda>0$, $R_{\lambda}^1 g = R_{\lambda}^2 g$ for all $g \in L^{\infty}(\R^d)$ such that $\|g\|_{L^{\infty}}(\mathbb R^d) \leq 1$, whence uniqueness follows under the assumption, which is subsequently removed by the localization argument.


\section{Proof of Theorem~\ref{thm}}\label{pf}

In this section, we will prove Theorem~\ref{thm} assuming some key propositions that will be proved in subsequent sections. We will prove that a solution to the martingale problem exists in Section~\ref{sec:exist}. Finally, we will prove that the solution is unique in Section~\ref{sec:uniq}. Throughout this section, let $\mathcal{L}$ be an operator as in \eqref{gsl} with kernel $K$ which satisfies Assumption~\ref{ass}.

\subsection{Proof of the Existence in Theorem~\ref{thm}}\label{sec:exist}

We require the following key proposition to prove the existence of a solution.

\begin{proposition}\label{prop:cty}
For all $f \in C_b^3(\mathbb R^d)$, $\mathcal Lf : \mathbb R^d \to \mathbb R$ is uniformly continuous.
\end{proposition}

Proposition~\ref{prop:cty} will be proved in Section~\ref{sec:cty}. We can now prove the existence as in Theorem~\ref{thm}.

\begin{proof}[Proof of the existence in Theorem~\ref{thm}]

Fix $x_0 \in \mathbb R^d$. Recall $\Omega = D([0,\infty); \mathbb R^d)$ from Section~\ref{mr}. For all $n \geq 1$, let $\P_n$ be a probability measure on $\Omega$ such that for all $k < n 2^n$ and for all $f \in C_b^2(\mathbb R^d)$,
\begin{equation}\label{eq:partmart}
M^{f,n}_t = f\left(X_{t \wedge \frac{k+1}{2^n}}\right)- f\left(X_{t \wedge \frac{k}{2^n}}\right) - \mathcal Lf(X_{\frac{k}{2^n}})\left(t \wedge \frac{k+1}{2^n} - t \wedge \frac{k}{2^n}\right)
\end{equation}
is a $\P_n$-martingale, and $\P_n(X_0=x_0)=1$. As described in the proof of \cite[Theorem 2.1]{B}, this martingale can be constructed using a stochastic differential equation driven by a Poisson point process. 

Fix $x \in \R^d$ and $f \in C_b^2(\R^d)$. Note that for any $h \in \R^d, \|h\|<1$ we have $$
|f(x+h)-f(x) - \nabla f(x)\cdot h| \leq \|H_f\|_{\infty}\|h\|^2.
$$
On the other hand, for $\|h\| \geq 1$ we have $$
|f(x+h)-f(x)| \leq 2\|f\|_{\infty}.
$$
Combining the two inequalities above,
\begin{align*}
|f(x+h)-f(x) - 1_{\|h\|<1} (\nabla f(x)\cdot h)|& \leq (2\|f\|_{\infty} \wedge \|H_f\|_{\infty}\|h\|^2)\\& \leq 2(1 \wedge \|h\|^2) (\|f\|_{\infty} \vee \|H_f\|_{\infty})
\end{align*}
for all $h \in \R^d \setminus \{0\}$. Thus, by \eqref{gsl} and Assumption~\ref{ass}(c), for all $x \in \mathbb R^d$ we have
\begin{align}
\|\mathcal Lf(x)\|& \leq 2(\|f\|_{\infty} \vee \|H_f\|_{\infty})\int_{\R^d \setminus \{0\}} (1 \wedge \|h\|^2)K(x,h)\dd h \nonumber \\ &\leq C_f\int_{\R^d \setminus \{0\}} (1 \wedge \|h\|^2)J(h)\dd h \leq 2C_fK_0  \label{eq:Lbdd} 
\end{align}
where we used Assumption~\ref{ass}(a) in the last inequality. It follows from the definition of $\P_n$ that $f(X_t) - f(X_0) - 2C_fK_0t$ is a $\P_n$- supermartingale for all $n \geq 1$. Thus, $\{\mathbb P_n\}_{n \geq 1}$ satisfy the hypothesis of \cite[Proposition 3.2]{B} since $C_f$ depends only on $\|f\|_{\infty}$ and $\|H_f\|_{\infty}$. It follows that the sequence of probability measures $\{\P_n\}_{n \geq 1}$ are tight on $D([0,t_0]; \mathbb R^d)$ for all $t_0 >0$.  

Letting $\mathbb P_{x_0}$ be a subsequential limit of $\mathbb P_n$, for arbitrary fixed $f \in C_b^3(\mathbb R^d)$ we follow the argument of \cite[Theorem 2.1]{B} from the third paragraph onwards, and conclude that \cite[(3.3)]{B} holds for $f \in C_b^3(\mathbb R^d)$. That is, $$
M^f_t = f(X_t) - f(x_0) - \int_{0}^t \mathcal Lf(X_s) \dd s
$$
is a $\mathbb{P}_{x_0}$-martingale for all $f \in C_b^3(\mathbb R^d)$. Now, given $f \in C_b^2(\mathbb R^d)$ arbitrary, let $f_n \in C_b^3(\mathbb R^d)$ be such that $f_n \to f$ in $C_b^2(\mathbb R^d)$. 

We claim that $M^{f_n}_t \to M^f_t$ in $L^1(\mathbb P_{x_0})$ for each fixed $t >0$. Since $f_n \to f$ in $C_b^2(\mathbb R^d)$ we have
$$
f_n(x+h) - f_n(x) - 1_{\|h\|<1}(\nabla f_n(x)\cdot h) \to f(x+h) - f(x) - 1_{\|h\|<1}(\nabla f(x)\cdot h) 
$$
for all $x \in \mathbb R^d$ and $h \in \mathbb R^d \setminus \{0\}$. By \eqref{gsl}, \eqref{eq:Lbdd} and the dominated convergence theorem, it now follows that $\mathcal Lf_n(x) \to \mathcal Lf(x)$ for all $x \in \mathbb R^d$. Since $f_n \to f$ pointwise, it follows that $M^{f_n}_t \to M^f_t$ pointwise as random variables on $\Omega$. 

By \eqref{eq:partmart} and \eqref{eq:Lbdd}, $$|M^{f_n}_t| \leq \|f_n\|_{C_b^2(\mathbb R^d)}(2+2C_fK_0t)$$ for all $n \geq 1$. Note that $\|f_n\|_{C_b^2(\mathbb R^d)}$ is a convergent, hence bounded sequence. It follows by the bounded convergence theorem that $M^{f_n}_t \to M^f_t$ in $L^1(\mathbb P_{x_0})$.

Thus, $\{M^{f}_t\}_{t \geq 0}$ is the $L^1(\mathbb P_{x_0})$-pointwise limit of $\{M^{f_n}_t\}_{t \geq 0}$. Since $\{M^{f_n}_t\}_{t \geq 0}$ are martingales and $\{\mathcal F_t\}_{t \geq 0}$ is a complete filtration, it follows that $\{M^f_t\}_{t \geq 0}$ is also a martingale. We have shown that $\mathbb{P}_{x_0}$ is a solution to the martingale problem for $\mathcal{L}$ started at $x_0$ for each $x_0 \in \mathbb R^d$.

Suppose that $\mathcal{G}_{x_0}$ is the set of all solutions to the martingale problem for $\mathcal{L}$ started at $x_0$. Then, by the arguments in \cite[Section 4]{B}, it can be shown that $\mathcal{G}_{x_0}$ is compact in the space of probability measures $\mathbb P$ on $D([0,\infty);\R^d)$ such that $\mathbb P(X_0=x_0) = 1$. Hence, by the proofs in \cite[Chapter 14]{SV}, we obtain the existence of $\mathbb P_{x_0} \in \mathcal{G}_{x_0}$ for each $x_0 \in \R^d$ such that $\{\P_x\}_{x \in \R^d}$ is a strong Markov family of solutions to the martingale problem (see the paragraph below \cite[Remark 4.2]{BT} for a similar argument).
\end{proof}

\subsection{Proof of Uniqueness in Theorem~\ref{thm}}\label{sec:uniq}

By the previous section, we assume the existence of a strong Markov family of solutions to the martingale problem $\{\P_x\}_{x \in \mathbb R^d}$ with associated coordinate process $\{X_t\}_{t \geq 0}$ as described in Section~\ref{mr}. 

Recall the resolvent of a strong Markov process from \eqref{res}. The following lemma states some facts about the resolvent. We include a proof of it after the proof of uniqueness in Theorem~\ref{thm}.

\begin{lemma}[Resolvent properties]\label{rescontr} Let $\{Y_t\}_{t \geq 0}$ be a strong Markov process with resolvent $R_{\lambda}^{Y}$. Fix $\lambda > 0$.
\begin{enumerate}[label = (\alph*)]
\item For any $g \in L^{\infty}(\R^d)$, $(R^Y_{\lambda}g) \in L^{\infty}(\R^d)$ and 
$$
\|(R^Y_{\lambda} g)\|_{\infty} \leq \frac{1}{\lambda} \|g\|_{\infty}.
$$
\item Suppose that $\{Y_t\}_{t \geq 0}$ is a Lévy process. Then, for any $\lambda>0$ and $g \in C_b^2(\R^d)$, $R^{Y}_{\lambda} g \in C_b^2(\R^d)$.
\end{enumerate}
\end{lemma}

We shall now state our first key proposition, the continuity of the resolvent. 

\begin{proposition}[Continuity of resolvent] \label{lem:resreg} Let $R_{\lambda}$ be the resolvent of $\{X_t\}_{t \geq 0}$. For any $g \in L^{\infty}(\R^d)$ and $\lambda > 0$, $R_{\lambda}g$ is a continuous function.
\end{proposition}

For stating our second key proposition, we require a stronger assumption on $A(x,h)$ than Assumption~\ref{ass}(e)(ii).

\begin{assumption}\label{xias}
There exists $\xi>0$ such that for every $x,y \in \R^d$ and $h \in \R^d \setminus \{0\}$,
$$
|A(x,h) - A(y,h)| \leq \frac{\xi}{\psi(\|h\|)},
$$
where $\psi$ is as in Assumption~\ref{ass}(e).
\end{assumption}
For $z \in \mathbb R^d$, let $\mathcal{M}^z$ be defined by 
\begin{align}\label{mzsimpl}
\mathcal M^zf(x) &= \int_{\mathbb R^d \setminus \{0\}} \left(f(x+h)-f(x)- (\nabla f(x) \cdot h)1_{\|h\| < 1}\right)K(z,h)\dd h \nonumber 
\\ & =\int_{\mathbb R^d \setminus \{0\}} \left(f(x+h)+f(x-h)- 2f(x)\right)K(z,h)\dd h
\end{align}
for all $f\in C_b^2(\R^d)$ (where we used Assumption~\ref{ass}(b) to derive the second equality). By \cite[Theorem 31.5,Chapter 6]{SAT}, there exists a unique symmetric Lévy process $\{X^z_t\}_{t \geq 0}$ which solves the martingale problem corresponding to the operator $\mathcal{M}^z$.

Let $R_{\lambda}^z$ denote the resolvent of $\{X^z_t\}_{t \geq 0}$. We are now ready to state our second key proposition, an estimate of how well $\mathcal M^{x_0}$ approximates $\mathcal{L}$. Note that if $f \in C_b^2(\mathbb R^d)$, then by Lemma~\ref{rescontr}(b), $R_{\lambda}^{x}f \in C_b^2(\mathbb R^d)$ for all $x \in \mathbb R^d, \lambda>0$. Therefore, $\mathcal{M}^{x}R_{\lambda}^{x}f$ is well defined for all $x \in \mathbb R^d$ and $\lambda>0$.

\begin{proposition}\label{rpb}
Suppose that $K$ satisfies Assumptions~\ref{ass} and \ref{xias}. For every $f \in C_b^2(\mathbb R^d)$, $x_0 \in \mathbb R^d$ and $\lambda >0$,
\begin{equation*}
\|(\mathcal L - \mathcal M^{x_0})(R_{\lambda}^{x_0}f)\|_{\infty} \leq \frac{4\xi\mathcal{J}}{\lambda} \|f\|_{\infty},
\end{equation*}
where $\mathcal{J}$ is as in Assumption~\ref{ass}(e)(iii) and $\xi$ is as in Assumption~\ref{xias}.
\end{proposition}

Our third key proposition is a localization argument.

\begin{proposition}[Localization]\label{prop:loc}
Suppose that every operator $\mathcal{L}$ of the form \eqref{gsl} such that $K$ satisfies Assumptions~\ref{ass} and \ref{xias} admits a unique solution to the martingale problem. Then, uniqueness holds in Theorem~\ref{thm}.
\end{proposition}

We are now ready to prove the main result of this section. This will be followed by the proof of Lemma~\ref{rescontr}.

\begin{proof}[Proof of Uniqueness in Theorem~\ref{thm}]

Let $\{\mathbb P^i_x\}_{x \in \mathbb R^d}, i=1,2$ be two strong Markov families of solutions to the martingale problem for $\mathcal L$. Recall the resolvent of a process defined by \eqref{res}. Consider the resolvents of the coordinate process $\{X_t\}_{t \geq 0}$ under each of these families of measures, \begin{equation}\label{sres}
R^i_{\lambda} f(x) = \mathbb E^i_x \int_0^{\infty} e^{-\lambda t}f(X_t)dt, \quad i=1,2.
\end{equation}
Let $R_{\lambda}^{\Delta}= R_{\lambda}^1 - R_{\lambda}^2$ denote their difference. Then, for all $f \in L^{\infty}(\R^d)$ and $i=1,2$,  $\|(R^i_{\lambda}f)\|_{\infty} \leq \frac{1}{\lambda} \|f\|_{\infty}$ by Lemma~\ref{rescontr}(a). By the triangle inequality,
\begin{equation}\label{eq:diff}
\|R^{\Delta}_{\lambda} f\|_{\infty} \leq \|R^{1}_{\lambda} f\|_{\infty} +\|R^{2}_{\lambda} f\|_{\infty} \leq \frac{2}{\lambda} \|f\|_{\infty}
\end{equation}
for all $f \in L^{\infty}(\R^d)$. Let $$
\mathcal{B} = \{g \in L^{\infty}(\R^d) : \|g\|_{\infty} \leq 1\}
$$
be the unit ball in $L^{\infty}(\R^d)$. Define \begin{equation}\label{Theta}\Theta = \sup_{g \in \mathcal{B}} \|R^{\Delta}_{\lambda} g\|_{\infty}.\end{equation} Note that $\Theta \leq \frac{2}{\lambda}$ by \eqref{eq:diff}.

Fix $g \in C_b^2(\mathbb R^d) \cap \mathcal{B}$ and $x,x_0 \in \mathbb R^d$. Recall that $R_{\lambda}^{x_0}$ is the resolvent of the unique process which solves the martingale problem corresponding to the operator $\mathcal{M}^z$, where $\mathcal{M}^z$ is given by \eqref{mzsimpl}. 

Following the computations in the proof of \cite[ Theorem 1.2, pages 1164-1165]{BT} and using the definition of the resolvent \eqref{res} and Lemma~\ref{rescontr}(b) where necessary, we have the identity
$$
(R^{\Delta}_{\lambda}g)(x) = R_{\lambda}^{\Delta}((\mathcal L - \mathcal M^{x_0})(R_{\lambda}^{x_0}g))(x).
$$
By Proposition~\ref{prop:loc}, it is sufficient to prove that uniqueness holds when $K$ also satisfies Assumption~\ref{xias}. We assume this from now on. Let $\lambda \geq 8 \xi \mathcal{J}$, where $\xi$ is as in Assumption~\ref{xias} and $\mathcal{J}$ is as in Assumption~\ref{ass}(e)(iii). Using the linearity of $R^{\Delta}_{\lambda}$, 
\begin{align}
(R^{\Delta}_{\lambda} g)(x) & = R_{\lambda}^{\Delta}((\mathcal L - \mathcal M^{x_0})(R_{\lambda}^{x_0}g))(x)\nonumber \\ & = \|(\mathcal L - \mathcal M^{x_0})(R_{\lambda}^{x_0}g)\|_{\infty} \left(R_{\lambda}^{\Delta}\left(\frac{(\mathcal L - \mathcal M^{x_0})(R_{\lambda}^{x_0}g)}{\|(\mathcal L - \mathcal M^{x_0})(R_{\lambda}^{x_0}g)\|_{\infty}}\right)\right)(x) \label{eq:res1}
\end{align}
By Proposition~\ref{rpb} and noting that $\lambda \geq 8 \xi \mathcal{J}$ and $\|g\| \leq 1$, \begin{equation}
\label{eq:res2}
 \|(\mathcal L - \mathcal M^{x_0})(R_{\lambda}^{x_0}g)\|_{\infty} \leq \frac{4 \xi \mathcal{J}}{\lambda} \|g\|_{\infty} \leq \frac 12.
\end{equation} 
On the other hand, by the definition \eqref{Theta} of $\Theta$, $$\frac{(\mathcal L - \mathcal M^{x_0})(R_{\lambda}^{x_0}g)}{\|(\mathcal L - \mathcal M^{x_0})(R_{\lambda}^{x_0}g)\|_{\infty}} \in \mathcal{B} \implies 
\left(R_{\lambda}^{\Delta}\left(\frac{(\mathcal L - \mathcal M^{x_0})(R_{\lambda}^{x_0}g)}{\|(\mathcal L - \mathcal M^{x_0})(R_{\lambda}^{x_0}g)\|_{\infty}}\right)\right)(x) \leq \Theta.
$$
Combining the above estimate with \eqref{eq:res1} and \eqref{eq:res2}, it follows that 
\begin{equation}\label{ext}
\|R^{\Delta}_{\lambda} g\|_{\infty} \leq \frac 12 \Theta
\end{equation}
for all $g \in C_b^2(\mathbb R^d) \cap \mathcal{B}$ and $\lambda \geq 8 \xi \mathcal{J}$.

For any $g \in \mathcal{B}$, we can find a sequence $\{g_n\}_{n \geq 1} \subset C_b^2(\mathbb R^d) \cap \mathcal{B}$ such that $g_n(x) \to g(x)$ a.e. $x \in \mathbb R^d$. By \eqref{sres} and the dominated convergence theorem, $(R^{\Delta}_{\lambda} g_n)(x) \to (R^{\Delta}_{\lambda} g)(x)$ for a.e. $x \in \mathbb R^d$. Applying \eqref{ext} to each $g_n, n \geq 1$ and letting $n \to \infty$, it follows that \eqref{ext} holds for all $g \in \mathcal{B}$. Taking the supremum over all $g \in \mathcal{B}$  in \eqref{ext} and using the definition \eqref{Theta} of $\Theta$, we obtain $\Theta \leq \frac 12 \Theta$. Since $\Theta < \infty$, it must be that $\Theta = 0$.

Therefore, $R^{\Delta}_{\lambda} g = 0$ a.e. for all $g \in \mathcal{B}$. By Proposition~\ref{lem:resreg}, $R^{\Delta}_{\lambda} g$ is continuous. Hence, $R^{\Delta}_{\lambda} g = 0$ identically for all $g \in \mathcal{B}$. By definition of $R^{\Delta}_{\lambda}$, $$
R^{1}_{\lambda} g = R^{2}_{\lambda} g \quad \text {for all } g \in \mathcal{B}, \lambda \geq 8 \xi \mathcal{J}.
$$
By the definition \eqref{sres} of $R^i_{\lambda}$, uniqueness of the Laplace transform and the right continuity of $\{X_t\}_{t \geq 0}$, it follows that \begin{equation*}
\mathbb E^1_x g(X_t) = \mathbb E^2_x g(X_t),
\end{equation*}
for all $x \in \mathbb R^d$, $t>0$, and continuous $g \in \mathcal{B}$. The same equality also follows for all $g \in \mathcal{B}$ by limiting arguments.

In particular, by setting $g = 1_A$ for any Borel $A \subset \mathbb R^d$ in the above equation, it follows that $\mathbb P^1_x(X_t \in A) = \mathbb P^2_x(X_t \in A)$ for all $x \in \mathbb R^d, t>0$ i.e. the one-dimensional distributions of $X_t$ are the same under $\mathbb P^i_x, i=1,2$ for all $x \in \mathbb R^d$ and $t>0$. Thus, applying \cite[(a), Theorem 4.2, Chapter 4]{EK} to the families $\{\mathbb P^i_x\}_{x \in \mathbb R^d}, i=1,2$, it follows that for each $x \in \mathbb R^d$, $\{X_t\}_{t \geq 0}$ has the same finite dimensional distributions under each of $\mathbb P^i_x, i=1,2$. Thus, uniqueness holds for the martingale problem, as desired.
\end{proof}

We will now prove Lemma~\ref{rescontr}.

\begin{proof}[Proof of Lemma~\ref{rescontr}]
We first prove part (a). Let $g \in L^{\infty}(\R^d)$ be given. For any $x \in \mathbb R^d$, by the definition \eqref{res} of the resolvent, \begin{align*}
(R^Y_{\lambda}g)(x) & = \mathbb E_x \int_{0}^{\infty} e^{-\lambda t}g(Y_t)dt \\ & \leq \|g\|_{\infty}\mathbb E_x \int_{0}^{\infty} e^{-\lambda t}dt \\
& = \frac 1{\lambda}\|g\|_{\infty}.
\end{align*}
Part (a) follows. We now prove part(b). Suppose that $\{Y_{t}\}_{t \geq 0}$ is a Lévy process, and let $f \in C_b^2(\mathbb R^d)$ be given. For any $x \in \mathbb R^d$ and $h \in \mathbb R^d \setminus \{0\}$, by the definition \eqref{res} of the resolvent $R^{Y}_{\lambda}$ we have \begin{equation}\label{res1}
R^{Y}_{\lambda} f(x+h) - R^{Y}_{\lambda} f(x) = \mathbb E_{x+h} \int_{0}^{\infty} e^{-\lambda t} f(Y_t) \dd t - \mathbb E_{x} \int_{0}^{\infty} e^{-\lambda t} f(Y_t) \dd t.
\end{equation}
Since $\{Y_{t}\}_{t \geq 0}$ is a Lévy process, the process $\{Y_t+h\}_{t \geq 0}$ under the measure $\mathbb P_x$ has the same distribution as the process $\{Y_{t}\}_{t \geq 0}$ under the measure $\mathbb P_{x+h}$. Therefore, $$
 \mathbb E_{x+h} \int_{0}^{\infty} e^{-\lambda t} f(Y_t) \dd t = \mathbb E_{x} \int_0^{\infty} e^{-\lambda t} f(Y_t+h) \dd t.
$$
Combining \eqref{res1} with the above equation and dividing by $\|h\|$, \begin{equation}\label{res2}
\frac{R^{Y}_{\lambda} f(x+h) - R^{Y}_{\lambda} f(x)}{\|h\|} = \mathbb E_{x} \int_0^{\infty} e^{-\lambda t}\frac{f(Y_t+h) - f(Y_t)}{\|h\|} \dd t.
\end{equation}
By the mean value theorem, $$
\int_0^{\infty} e^{-\lambda t}\frac{f(Y_t+h) - f(Y_t)}{\|h\|} \dd t \leq \|f'\|_{\infty} \int_0^{\infty} e^{-\lambda t} \dd t \leq \frac{\|f'\|_{\infty}}{\lambda}.
$$

By \eqref{res2}, the above inequality and the dominated convergence theorem, it follows that $$\frac{\partial (R^{Y}_{\lambda}f)}{\partial x_i} = R^{Y}_{\lambda} \left(\frac{\partial f}{\partial x_i}\right),$$ for all $i=1,\ldots,d$. By part(a), $$ \left\|\frac{\partial (R^{Y}_{\lambda}f)}{\partial x_i}\right\|_{\infty} = \left\| R^{Y}_{\lambda} \left(\frac{\partial f}{\partial x_i}\right)\right\|_{\infty} \leq \frac 1{\lambda}\left\|\frac{\partial f}{\partial x_i}\right\|_{\infty}$$ for all $i=1,\ldots,d$. Iterating this argument for higher derivatives, it follows that $(R^{Y}_{\lambda} f) \in C_b^2$.
\end{proof}

\section{Proof of Proposition~\ref{prop:cty}}\label{sec:cty}

In this section, we prove Proposition~\ref{prop:cty}. Throughout this section, let $\mathcal{L}$ be as in \eqref{gsl} with kernel $K$ of the form $K(x, h) = A(x, h)J(h)$, satisfying Assumption~\ref{ass}.

\begin{proof}[Proof of Proposition~\ref{prop:cty}]
Fix $\epsilon>0$ and $f \in C_b^3(\mathbb R^d)$. Let
\begin{equation}\label{gdefn}
g(x,h) = (f(x+h)- f(x) - 1_{\|h\|<1} (\nabla f(x)\cdot h))K(x,h)
\end{equation}
for all $x \in R^d, h \in \mathbb R^d \setminus \{0\}$. By \eqref{gsl}, 
\begin{equation}\label{Lg}
\mathcal{L}f(y) = \int_{\mathbb R^d \setminus \{0\}} g(y,h)\dd h \quad \text{for all } y \in \mathbb R^d.
\end{equation}
Observe that by Assumption~\ref{ass}(d), for some constant $C>0$ we have \begin{equation}\label{eq:gbdd}
|g(x,h)| \leq C(\|f\|_{\infty} \wedge \|H_f\|_{\infty}) (1 \wedge \|h\|^2) J(h)
\end{equation}
for all $x \in \R^d, h \in \R^d \setminus \{0\}$. Let $M>0$ be a constant that will be chosen later. For any $\delta>0$, by \eqref{eq:gbdd}, 
$$
\left|\int_{\|h\|>\delta^{-1}} g(x,h) \dd h\right| \leq C\int_{\|h\|> \delta^{-1}} (1\wedge \|h\|^2) J(h) \dd h.
$$
Since $ (1\wedge \|h\|^2) K(0,h)$ is integrable over $\mathbb R^d \setminus \{0\}$ by Assumption~\ref{ass}(a), we can choose $\delta_M$ small enough depending on $M$ such that $$
\int_{\|h\|\geq \delta_M^{-1}} |g(x,h)| \dd h  < \frac{\epsilon}{4M}
$$
for all $x \in \mathbb R^d$. 

So, for any $x,y \in \mathbb R^d$, by \eqref{Lg},  the triangle inequality and the above inequality we have
\begin{align}\label{eq:trimmed}
|\mathcal Lf(x) - \mathcal Lf(y)| &\leq \int_{\|h\| \geq \delta_M^{-1}} (|g(x,h)| + |g(y,h)|) \dd h + \int_{\|h\|< \delta_M^{-1}} |g(x,h) - g(y,h)| \dd h \nonumber \\
& \leq \frac{\epsilon}{2M} + \int_{\|h\|< \delta_M^{-1}} |g(x,h) - g(y,h)| \dd h.
\end{align}

Let $\bar{g}(x,h) = f(x+h) - f(x) - 1_{\|h\|<1} \nabla (f(x) \cdot h)$. Note that by the triangle inequality,
\begin{align}
\int_{\|h\|< \delta_M^{-1}} |g(x,h) - g(y,h)| \dd h \leq& \int_{\|h\| < \delta_M^{-1}} \left|\bar{g}(x,h) - \bar{g}(y,h)\right|K(x,h) \dd h\nonumber \\
+ & \int_{\|h\|< \delta_M^{-1}} |\bar{g}(y,h)|\left|K(x,h) - K(y,h)\right|\dd h \label{eq:trimbd}.
\end{align}
We shall now bound the first term in \eqref{eq:trimbd}. Observe that for some constant $C>0$, 
$$
\left|\bar{g}(x,h) - \bar{g}(y,h)\right| \leq C \|f\|_{C_b^3(\mathbb R^d)} (1 \wedge \|h\|^2)
$$
for all $x \in \R^d,h \in \R^d \setminus \{0\}$. So, by parts (a) and (d) of Assumption~\ref{ass}, \begin{align}
&\int_{\|h\| < \delta_M^{-1}} \left|\bar{g}(x,h) - \bar{g}(y,h)\right|K(x,h) \dd h \nonumber \\ \leq & \|f\|_{C_b^3(\mathbb R^d)} \|x-y\|\int_{\|h\| < \delta_M^{-1}} (1 \wedge \|h\|^2)K(x,h) \dd h \nonumber \\  \leq & c_2\|f\|_{C_b^3(\mathbb R^d)} \|x-y\|\int_{\|h\| < \delta_M^{-1}} (1 \wedge \|h\|^2)J(h) \dd h \\
& \leq c_2\|f\|_{C_b^3(\mathbb R^d)}K_0 \|x-y\|\label{eq:trimbd1}.
\end{align}

We will now bound the second term in \eqref{eq:trimbd}. By Assumption~\ref{ass}(e)(ii) applied with $b = \delta_M^{-1}$, there exists $\delta_M'>0$ depending on $M$ such that 
\begin{equation}\label{eq:deltam}
\|x-y\| < \delta_M' \implies \sup_{\|h'\| < \delta_M^{-1}}|A(x,h') - A(y,h')| \psi(h') < \frac{\epsilon}{2M}.
\end{equation}
Thus, whenever $\|x-y\| < \delta_M'$, by Assumption~\ref{ass}(d) and the above inequality, for all $\|h\| < \delta_M^{-1}$ we have
\begin{align*}
|K(x,h) - K(y,h)| &= |A(x,h) - A(y,h)| J(h) \nonumber \\ &\leq \left(\sup_{\|h'\| < \delta_M^{-1}}|A(x,h') - A(y,h')| \psi(h')\right) \frac{J(h)}{\psi(h)} \nonumber\\  & < \frac{\epsilon}{2M} \frac{J(h)}{\psi(h)}.
\end{align*}

Note that $|\bar{g}(x,h)| \leq 2 \|f\|_{C_b^2(\mathbb R^d)}$. Combining this with the above estimate, \begin{equation}\label{eq:trimbd2}
\int_{\|h\| < \delta_M^{-1}} |\bar{g}(y,h)| |K(x,h) - K(y,h)| \dd h \leq \frac{\epsilon}{M}\|f\|_{C_b^2(\mathbb R^d)} \int_{\|h\| < \delta_M^{-1}} \frac{J(h)}{\psi(h)} \dd h \leq \frac{\mathcal{J}\|f\|_{C_b^2(\mathbb R^d)} \epsilon}{M},
\end{equation}
where we used Assumption~\ref{ass}(e)(iii) in the last inequality. Combining \eqref{eq:trimbd}, \eqref{eq:trimbd1} and \eqref{eq:trimbd2} we have $$
\int_{\|h\|< \delta_M^{-1}} |g(x,h) - g(y,h)| \dd h \leq \left(\mathcal{J}\|f\|_{C_b^2(\mathbb R^d)}\right)\frac{\epsilon}{M} + \|f\|_{C_b^3(\mathbb R^d)} K_0\|x-y\|
$$
Combining this with \eqref{eq:trimmed}, \begin{equation}\label{eq:finish}
|\mathcal Lf(x) - \mathcal Lf(y)| \leq \left(\mathcal{J}\|f\|_{C_b^2(\mathbb R^d)}+\frac 12\right)\frac{\epsilon}{M} + c_2\|f\|_{C_b^3(\mathbb R^d)} K_0\|x-y\|
\end{equation}
for all $x,y \in \mathbb R^d$ such that $\|x-y\| \leq \delta_M'$. In the previous argument, we now choose $$M = \frac 1{2\left(\mathcal{J}\|f\|_{C_b^2(\mathbb R^d)}+\frac 12\right)}$$ and obtain $\delta_M'>0$ such that \eqref{eq:deltam} holds for all $\|x-y\| \leq \delta_{M}'$. Furthermore, let $$\delta'' = \frac {\epsilon}{2c_2\|f\|_{C_b^3(\mathbb R^d)} K_0}.$$ Then it follows from \eqref{eq:finish} that $$
\|x-y\| \leq \delta_M' \wedge \delta'' \implies |\mathcal Lf(x) - \mathcal Lf(y)| < \epsilon,
$$
which proves that $\mathcal Lf$ is uniformly continuous, as desired.
\end{proof}

\section{Proof of Proposition~\ref{lem:resreg}}\label{pfresreg}

In this section, we will prove Proposition~\ref{lem:resreg}. Let $\mathcal{L}$ be as in \eqref{gsl} with kernel $K$ satisfying Assumption~\ref{ass}. Suppose that $\{\P_x\}_{x \in \mathbb R^d}$ is a strong Markov family of solutions to the martingale problem and $\{X_t\}_{t \geq 0}$ is the associated coordinate process as in Section~\ref{mr}.

We will require two additional results for the proof. Define $L : (0,1) \to (0,\infty)$ by \begin{equation*} L(r) = \int_r^1 \frac{l(s)}{s}ds,\end{equation*} 
where $l$ is as in Assumption~\ref{ass}(c). By the same assumption,
\begin{equation}\label{Lprop}
\lim_{r \to 0^+} L(r) = +\infty.
\end{equation}
For $x \in \mathbb R^d, r>0$ let $B(x,r) = \{y \in \R^d : \|y-x\|<r\}$ be the Euclidean ball around $x$ of radius $r$, and let $\tau_D = \inf\{t >0 : X_t \notin D\}$ be the exit time of $\{X_t\}_{t \geq 0}$ from a set $D\subset \R^d$. We have the following upper bound on the exit time of $\{X_t\}_{t \geq 0}$ from small balls.

\begin{lemma}[Exit time estimate]\label{ete}
There exists $C>0$ such that 
$$
\mathbb E_z \tau_{B(x,r)} \leq \frac{C_{2}}{L(r)}
$$
for all $r \in (0,1)$, $x \in \mathbb R^d$ and $z \in B(x,r)$. 
\end{lemma}
\begin{proof}
cf. \cite[Proposition 3.3]{MKNL}.
\end{proof}

A function $f : \R^d \to \R$ is harmonic on $D \subset \mathbb R^d$ if, for all $B\subset D$ open and $x \in B$ we have $\mathbb E_x[h(X_{\tau_B})] = h(x)$. We also have the following regularity estimate for bounded harmonic functions.
\begin{lemma}[Regularity of harmonic functions]\label{lem:reghar}
There exists $c >0$ and $\gamma \in (0,1)$ such that for all $r \in (0,\frac 12)$ and $x_0 \in \R^d$,
$$
|u(x)-u(y)| \leq c \|u\|_{\infty} \frac{L(\|x-y\|)^{-\gamma}}{L(r)^{-\gamma}}, \quad x,y \in B(x_0, r/4)
$$
for all bounded $u : \R^d \to \R$ that are harmonic in $B(x_0,r)$.
\end{lemma}
\begin{proof}
cf. \cite[Theorem 1.4]{MKNL}.
\end{proof}

We are now ready to prove the proposition.

\begin{proof}[Proof of Proposition~\ref{lem:resreg}]
Fix $\lambda > 0$ and $h \in L^{\infty}(\R^d)$. Fix $x \in \mathbb R^d$ and let $r_0 \in (0,\frac 12)$ be a constant that will be chosen later. We follow the argument as in \cite[Proposition 4.2]{BL}.

Begin with the identity
$$
(R_0h)(z) = \mathbb E^z \int_0^{\tau_{B(x,r_0)}} h(X_s) \dd s + \mathbb E^z R_0h(X_{\tau_{B(x,r)}})
$$
which holds for all $z \in B(x,r_0)$. Let $y \in B(x, \frac{r_0}{4})$ be arbitrary. We plug in $z=x$ and $z=y$ in the above identity, take the difference and use the triangle inequality to get
\begin{align*} 
\left|(R_0h)(x) - (R_0h)(y)\right| \leq &\left|\mathbb E^x \int_0^{\tau_{B(x,r_0)}} h(X_s) \dd s\right|+\left|\mathbb E^y \int_0^{\tau_{B(x,r_0)}} h(X_s)\dd s\right|\\ + & |\mathbb E^x R_0h(X_{\tau_{B(x,r_0)}}) -  \mathbb E^y R_0h(X_{\tau_{B(x,r_0)}})|
\end{align*}
The first two terms may be bounded by $\|h\|_{\infty} \sup_{z \in B(x,r_0)} \mathbb E_z\tau_{B(x,r_0)}$. Thus, we see that for any $y \in B(x,\frac {r_0}4)$,
\begin{align}\label{resreg11}
&\left|(R_0h)(x) - (R_0h)(y)\right| \nonumber
\\  \leq & 2 \|h\|_{\infty} \sup_{z \in B(x,r_0)} \mathbb E_z\tau_{B(x,r_0)} + \left|\mathbb E_x\left[(R_0h)(X_{\tau_{B(x,r_0)}})\right] - \mathbb E_y\left[(R_0h)(X_{\tau_{B(x,r_0)}})\right]\right|. 
\end{align}
Observe that the function $z \to \mathbb E^z\left[(R_0h)(X_{\tau_{B(x,r_0)}})\right]$ is bounded and harmonic in $B(x,r_0)$. Since $r_0 \in (0,\frac 12)$, we can apply  Lemma~\ref{lem:reghar} to this function. Thus, there exist constants $C>0,\gamma \in (0,1)$ such that \begin{equation*}
 \left|\mathbb E_x\left[(R_0h)(X_{\tau_{B(x,r_0)}})\right] - \mathbb E_y\left[(R_0h)(X_{\tau_{B(x,r_0)}})\right]\right| \leq C\|h\|_{\infty}\left(\frac{L(r_0)}{L(\|x-y\|)}\right)^{\gamma},
\end{equation*}
for all $y \in B(x,\frac {r_0}4)$. Combining the above with \eqref{resreg11},
\begin{equation}\label{add1}
\left|(R_0h)(x) - (R_0h)(y)\right| \leq 2 \|h\|_{\infty} \sup_{z \in B(x,r_0)} \mathbb E_z\tau_{B(x,r_0)}+C\|h\|_{\infty}\left(\frac{L(r_0)}{L(\|x-y\|)}\right)^{\gamma}
\end{equation}
for all $y\in B(x_0,\frac{r_0}{4})$. 

Now, let $g \in L^{\infty}$ have compact support and let $h= g - \lambda R_{\lambda} g$. By the triangle inequality and Lemma~\ref{rescontr}(a), $\|h\|_{\infty}\leq 2\|g\|_{\infty}$. Let $\{X'_{t}\}_{t \geq 0}$ be an independent copy of $\{X_t\}_{t \geq 0}$. We will now prove that $R_{\lambda} g = R_0h$. To prove this, fix $x \in \R^d$ and let $$
k(s) = \int_s^{\infty} \mathbb E_{x}g(X_t)\dd t.
$$
We have $k(0) = R_{0}g(x)$, and $k'(s) = -\mathbb E_{x}g(X_s)$. Using integration by parts, \begin{equation}\label{eq:ibp}
\int_{0}^{\infty} \lambda e^{-\lambda s} k(s) \dd s = - R_{0}g(x) + \int_{0}^{\infty} e^{-\lambda s} \mathbb E_{x}g(X_s)\dd s = -R_{0}g(x) + R_{\lambda}g(x). 
\end{equation}
Let $\{X'_t\}_{t \geq 0}$ be an independent copy of $\{X_t\}_{t \geq 0}$. Then, \begin{align}
(R_0(h-g))(x) = (R_0 (\lambda R_{\lambda})g)(x) & = \mathbb E_x \int_0^{\infty} \lambda (R_{\lambda} g)(X_t) \dd t \nonumber  \\
&= \mathbb E_x \int_0^{\infty} \lambda \mathbb E_{X_t} \int_0^{\infty}  e^{-\lambda s}g(X'_s) \dd s \dd t \nonumber \\
& = \int_0^{\infty} \lambda e^{-\lambda s}\int_0^{\infty} [\mathbb E_x  \mathbb E_{X_t}g(X'_s)] \dd t \dd s \nonumber\\
& =  \int_0^{\infty} \lambda e^{-\lambda s}\int_0^{\infty} \mathbb E_x g(X_{t+s}) \dd t \dd s \nonumber\\
 &=   \int_0^{\infty} \lambda e^{-\lambda s} k(s)\dd s,\label{eq:smp} 
\end{align}
where we used the Strong Markov property of $\{X_t\}_{t \geq 0}$ in the second last equality.  Combining \eqref{eq:ibp} and \eqref{eq:smp} it follows that $R_{\lambda}g = R_0h$. Now, by \eqref{add1},
\begin{equation}\label{resreg1}
\left|(R_{\lambda}g)(x) - (R_{\lambda}g)(y)\right| \leq 4 \|g\|_{\infty}  \sup_{z \in B(x,r_0)} \mathbb E_z\tau_{B(x,r_0)} + 2C\|g\|_{\infty}\left(\frac{L(r_0)}{L(\|x-y\|)}\right)^{\gamma},
\end{equation}
for all $y \in B(x,\frac {r_0}4)$. 

Observe that \eqref{resreg1} holds for all bounded $g$ with compact support. However, any bounded measurable $g$ can be approximated pointwise from below by a sequence of bounded functions $g_n,n\geq 1$ which have compact support. By the dominated convergence theorem, $R_{\lambda} g_n \to R_{\lambda}g$ pointwise. Applying \eqref{resreg1} to $g_n$ for each $n\geq 1$ and letting $n \to \infty$, we see that \eqref{resreg1} holds for all bounded measurable $g$. 

Now, fix $g \in L^{\infty}(\R^d)$ and let $\epsilon>0$ be given. By Lemma~\ref{ete} and \eqref{Lprop}, $$\lim_{r \to 0} \sup_{z \in B(x,r)} \mathbb E_z\tau_{B(x,r)} = 0.$$ We choose $r_0$ small enough such that \begin{equation}\label{resreg1.5}\sup_{z\in B(x,r_0)} \mathbb E_z\tau_{B(x,r_0)} < \frac{\epsilon}{8\|g\|_{\infty}}.\end{equation} By \eqref{Lprop}, we can choose $0<r'<r_0$ small enough such that \begin{equation*}
\sup_{z\in B(x,r')} \left(\frac{L(r_0)}{L(\|x-z\|))}\right)^{\gamma} =  \left(\frac{L(r_0)}{L(r')}\right)^{\gamma} < \frac{\epsilon}{4C\|g\|_{\infty}}. 
\end{equation*}
Combining \eqref{resreg1}, \eqref{resreg1.5} and the above equation,
\begin{equation*}
\left|R_{\lambda}g(x) - R_{\lambda}g(y)\right| < \epsilon \quad \text{for all } y \in B(x,r').
\end{equation*}
Thus, $R_{\lambda}g$ is a continuous function.
\end{proof}

\section{Proof of Proposition~\ref{rpb}} \label{pfrpb}

In this section, we will prove Proposition~\ref{rpb}, which is required for the proof of uniqueness in Theorem~\ref{thm}. Throughout, we assume that $\mathcal{L}$ is as in \eqref{gsl} with kernel $K$ of the form $K(x,h) = A(x,h)J(h)$ satisfying Assumptions~\ref{ass} and \ref{xias}.

\begin{proof}[Proof of Proposition~\ref{rpb}]
Fix $x_0 \in \mathbb R^d$, $\lambda>0$ and $f\in C_b^2(\R^d)$. By Lemma~\ref{rescontr}(b), $R_{\lambda}^{x_0}f \in C_b^2(\mathbb R^d)$. Hence, it lies in the domains of $\mathcal L$ and $\mathcal M^{x_0}$. By \eqref{gsl} and \eqref{mzsimpl} we have
\begin{align*}
&(\mathcal{L} - \mathcal M^{x_0})(R^{x_0}_{\lambda} f)(x) \\ =& \int_{\mathbb R^d \setminus \{0\}} ((R^{x_0}_{\lambda}f)(x+h) + (R^{x_0}_{\lambda}f)(x-h) - 2(R^{x_0}_{\lambda}f)(x))(K(x,h) - K(x_0,h))\dd h.
\end{align*}
Taking absolute values on both sides and applying the triangle inequality,
\begin{align}\label{rpbp1}
& |(\mathcal{L} - \mathcal M^{x_0})(R^{x_0}_{\lambda} f)(x)| \nonumber \\ \leq & \int_{\mathbb R^d \setminus \{0\}} \left|((R_{\lambda}^{x_0} f)(x+h) + (R_{\lambda}^{x_0}f)(x-h) - 2(R_{\lambda}^{x_0}f)(x))(K(x,h) - K(x_0,h))\right|\dd h  \nonumber \\
\leq & 4 \|R_{\lambda}^{x_0} f\|_{\infty}\int_{\mathbb R^d \setminus \{0\}} |K(x,h) - K(x_0,h)| \dd h \nonumber \\
 \leq & \frac{4}{\lambda} \|f\|_{\infty}\int_{\mathbb R^d \setminus \{0\}} |K(x,h) - K(x_0,h)| \dd h,
\end{align}
where we used Lemma~\ref{rescontr}(a) in the final inequality. By Assumption~\ref{xias}, 
$$
|K(x,h) - K(x_0,h)| = |A(x,h) - A(x_0,h)|J(h) \leq \xi \frac{J(h)}{\psi(\|h\|)},
$$
for all $h \in \mathbb R^d \setminus \{0\}$. Applying this bound to the right hand side of \eqref{rpbp1},
$$
\int_{\mathbb R^d \setminus \{0\}} |K(x,h) - K(x_0,h)| \dd h \leq \xi \int_{\mathbb R^d \setminus \{0\}} \frac{J(h)}{\psi(\|h\|)}\dd h = \xi \mathcal{J},
$$
where $\mathcal{J}$ is as in Assumption~\ref{ass}(e)(iii). Combining \eqref{rpbp1} with the above inequality, the proposition follows.
\end{proof}

\section{Proof of Proposition~\ref{prop:loc}}\label{sec:loc}

In this section, we prove Proposition~\ref{prop:loc}, the localization argument required to prove uniqueness in Theorem~\ref{thm}. Let $\mathcal{L}$ be as in \eqref{gsl} with kernel $K$ which satisfies Assumption~\ref{ass}. Let $l : (0,1) \to (0,\infty)$ be as in Assumption~\ref{ass}(c) and $\psi : \mathbb R_+ \to \mathbb R_+$ be as in Assumption~\ref{ass}(e). For $x \in \R^d, r>0$ let $B(x,r) = \{y \in \R^d : \|y-x\| < r\}$ be the Euclidean ball around $x$ of radius $r$.

We require some preliminary lemmas before proceeding to the proof of Proposition~\ref{prop:loc}. We will first create a family of localized operators for $\mathcal L$ which satisfy Assumption~\ref{xias}.

\begin{lemma}\label{lem:loc}
There exists $\eta>0$ and, for each $x_0 \in \mathbb R^d$, there exists an operator $\mathcal L_{x_0} : C_b^2(\mathbb R^d) \to  L^{\infty}(\mathbb R^d)$ of the form \eqref{gsl} such that :
\begin{enumerate}[label = (\alph*)]
\item For all $f \in C_b^2(\mathbb R^d)$ and $x \in B(x_0,\eta)$, $$
\mathcal Lf(x) = \mathcal L_{x_0}f(x).
$$
\item For all $x_0 \in \mathbb R^d$, the kernel of $\mathcal{L}_{x_0}$ satisfies Assumptions~\ref{ass} and \ref{xias}.
\end{enumerate}
\end{lemma}
\begin{proof}
By Assumption~\ref{ass}(e)(ii) applied with $b=\epsilon=1$, there exists $\delta>0$ such that for all $x_0 \in \mathbb R^d, 0<\|h\| \leq 1$ and $y \in B(x_0,\delta)$ we have 
\begin{equation}\label{eq:hl1}
|A(y,h) - A(x_0,h)| \leq \frac{1}{\psi(\|h\|)}.
\end{equation}
By the triangle inequality and Assumption~\ref{ass}(d), for all $\|h\| \geq 1$ and $y \in B(x_0,\delta)$ we have
$$
|A(y,h) - A(x_0,h)| \leq |A(y,h)| + |A(x_0,h)| \leq 2c_2
$$
By Assumption~\ref{ass}(e)(i), we have $\psi(s) = 1$ for $s \geq 1$. Combining this with the above inequality, for all $\|h\| \leq 1$ and $y \in B(x_0,\delta)$,
$$
|A(y,h) - A(x_0,h)| \leq 2c_2\leq \frac{2c_2}{\psi(\|h\|)}
$$
Combining the above inequality with \eqref{eq:hl1}, there exists $C>0$ such that for all $y \in B(x_0,\delta)$ and $h \in \R^d \setminus \{0\}$, 
$$
|A(y,h)| - |A(x_0,h)| \leq \frac{C}{\psi(\|h\|)}.
$$
For all $x,y \in B\left(x_0,\delta\right)$ and $h \in \R^d \setminus \{0\}$, we use the above inequality to see that
\begin{equation}\label{local1}
|A(y,h) - A(x,h)| \leq |A(y,h) - A(x_0,h)| + |A(x_0,h) - A(x,h)| \leq \frac{2C}{\psi(\|h\|)}.
\end{equation}

Fix $x_0 \in\mathbb R^d$ and let $\eta = \frac{\delta}{2}$. We shall now define $A_{x_0} : \mathbb R^d \times \mathbb R^d \setminus \{0\} \to \mathbb R_+$. For each $z$ such that $z \in B(x_0,\eta)^c$, let $\Lambda(z)$ be the unique point $y$ on the line segment joining $z$ and $x_0$ such that $\|y-x_0\| = \eta$. Define
$$
A_{x_0}(x,h) = \begin{cases}
A(x,h) & x \in B(x_0,\eta) \\
A(\Lambda(x),h) & B(x_0,\eta)^c
\end{cases}.
$$

Now, we define $K_{x_0}: \R^d \times \R^d \setminus \{0\} \to \R_+$ by $$K_{x_0}(x,h) = A_{x_0}(x,h)J(h).$$ By \eqref{local1} and the definition of $A_{x_0}$, $K_{x_0}$ satisfies Assumptions~\ref{ass}(c) with the function $l$, Assumption~\ref{ass}(e) with $\psi$, and Assumption~\ref{xias}. Furthermore, $$K_{x_0}(x,h) = K(x,h) \quad \text{for all } x \in B\left(x_0,\eta\right), h \in \mathbb R^d \setminus \{0\}.$$ Let $\mathcal L_{x_0}$ be given by \eqref{gsl} with kernel $K_{x_0}$. By \eqref{gsl} and the above facts about $K_{x_0}$, it is clear that $\mathcal L_{x_0}$ satisfies condition (a) and (b), completing the proof.
\end{proof}

In order to state our second lemma, let $\{\mathbb P_{x}\}_{x \in \mathbb R^d}$ be a strong Markov family of solutions to the martingale problem for $\mathcal L$. Let $\eta>0$ be given by Lemma~\ref{lem:loc} and define the stopping times $\tau_i : \Omega \to \mathbb R_+$ for $i\geq 0$ by \begin{equation}\label{eq:st}
\tau_0= 0, \tau_i = \inf\{t \geq \tau_{i-1} : \|X_{t} - X_{\tau_{i-1}}\| \geq \eta\} \text{ for } i \geq 1.
\end{equation}

Let $\mathcal F_{\tau_i}$ be the filtration We require the following fact about $\mathcal F_{\tau_i},i \geq 1$. Note that the space $(\Omega, \mathcal{F}_{\infty})$ is sufficiently rich in that, for any $t \geq 0$ and $\omega \in \Omega$, the element $\omega' \in \Omega$ given by $\omega'(s) = \omega(t \wedge s)$ satisfies \cite[(1.11)]{SH}. By \cite[Theorem 6]{SH} it follows that \begin{equation}\label{eq:sa}\mathcal{F}_{\tau_i} = \sigma(X_{t \wedge \tau_i} : t \geq 0) \quad \text{for all } i \geq 1.\end{equation}

\begin{lemma}\label{lem:st}
$\tau_i \to \infty$ a.s. under every $\mathbb P_x$, $x \in \mathbb R^d$. Furthermore, $\sigma(\cup_{i=1}^\infty \mathcal F_{\tau_i}) = \mathcal F_{\infty}$.
\end{lemma}

\begin{proof}
Let us assume that there exists a constant $C>0$ such that for all $x_0 \in \mathbb R^d$ and $t>0$, \begin{equation}\label{eq:stbd}
\mathbb P_{x_0}(\tau_1 \leq t) \leq Ct/\eta^2
\end{equation}
for some constant $C>0$ independent of $x_0$, $t$ and $\eta$. We will prove that $\tau_i \to +\infty$ a.s. assuming \eqref{eq:stbd}. Then, we will prove \eqref{eq:stbd}.

Fix $x_0 \in \mathbb R^d$. For any $s \geq 0$, by \eqref{eq:stbd}, $$
\mathbb P_{x_0}(e^{-\tau_1} \geq s) = \mathbb P_{x_0}(\tau_1 \leq - \ln s) \leq 0 \vee (C(-\ln(s))/ \eta^2 \wedge 1).
$$
By the layer-cake formula and the above bound, \begin{align*}
\mathbb E_{x_0}[e^{-\tau_1}] &= \int_{0}^1 \mathbb P_{x_0}(e^{-\tau_1} \geq s) \dd s \\ &\leq \int_0^{e^{-\eta^2/C}} 1 \dd s + \int_{e^{-\eta^2/C}}^{1} -\frac{C \ln(s)}{\eta^2} \dd s \\ & = e^{-\eta^2/C} + \frac{1 - \eta^2/C - e^{-\eta^2/C}}{\eta^2/C}.
\end{align*}
By increasing $C$ as much as necessary, we can therefore ensure that 
\begin{equation}\label{eq:stexpbd}
\mathbb E_{x_0}[e^{-\tau_1}]  \leq \gamma < 1
\end{equation}
for some $\gamma$ independent of $x_0$. Since $\tau_i \geq \tau_{i-1}$, we have by the Strong Markov property that 
$$
\mathbb E_{x_0}[e^{-\tau_i}] = \mathbb E_{x_0}[\mathbb E_{x_0}[e^{-\tau_{i-1} - (\tau_i - \tau_{i-1})} | \mathcal F_{\tau_{i-1}}]] = \mathbb E_{x_0}[e^{\tau_{i-1}}\mathbb E_{X_{\tau_{i-1}}} e^{-(\tau_i - \tau_{i-1})}]
$$
Applying this inductively with \eqref{eq:stexpbd} we have $\mathbb E_{x_0}[e^{-\tau_i}] \leq \gamma^i$ for all $x_0 \in \mathbb R^d, i \geq 1$.

It follows that $e^{-\tau_i} \to 0$ in $L^1(\Omega, \mathbb P_{x_0})$. However, $e^{-\tau_i}$ is an a.s. decreasing sequence of random variables since $\tau_i \leq \tau_{i+1}$ a.s. Hence, it follows that $e^{-\tau_i} \to 0$ a.s., and $\tau_i \to +\infty$ a.s..

We will now prove \eqref{eq:stbd} by using \cite[Proposition 3.1]{B}. Let $f \in C_b^2(\mathbb R^d)$ be fixed. Note that for all $x \in \R^d$ and $f \in C_b^2(\R^d)$, we have $$
|f(x+h)-f(x) - 1_{\|h\|<1} (\nabla f(x)\cdot h)| \leq 2 (\|f\|_{\infty} \wedge \|H_f\|_{\infty}).
$$
Thus, by \eqref{gsl} and Assumption~\ref{ass}, for all $x \in \mathbb R^d$,
\begin{align}
\|\mathcal Lf(x)\| \leq 2(\|f\|_{\infty} \wedge \|H_f\|_{\infty})\int_{\R^d \setminus \{0\}} K(x,h)\dd h &\leq 2C(\|f\|_{\infty} \wedge \|H_f\|_{\infty})\int_{\R^d \setminus \{0\}} J(h)\dd h \nonumber\\ & \leq 2CK_0 \nonumber
\end{align}
Therefore ,by \eqref{eq:mart} we have $$
M^f_t \leq f(X_t) - f(X_0) - C(\|f\|_{\infty} + \|H_f\|_{\infty})t.
$$
It follows that the right hand side is a $\mathbb P_{x_0}$-supermartingale. Now, \eqref{eq:stbd} is a direct consequence of \cite[Proposition 3.1]{B}.

We will now prove that $\sigma(\cup_{i=1}^\infty \mathcal F_{\tau_i}) = \mathcal F_{\infty}$. Clearly, $\sigma(\cup_{i=1}^\infty \mathcal F_{\tau_i}) \subset \mathcal F_{\infty}$ holds. Thus, we only need to show the reverse inclusion.

Let $t_1,\ldots,t_n \geq 0$ and $B_1,\ldots,B_n \in \mathcal B(\mathbb R^d)$ be arbitrary. Let $D = \cap_{1 \leq i \leq n} \{X_{t_i} \in B_i\}$ and $D_j = \cap_{1 \leq i \leq n} \{X_{t_i \wedge \tau_j} \in B_i\}$ for all $j \geq 1$. Since $\tau_j \to \infty$ a.s. it follows that $1_{D_j} \to 1_D$ a.s. as $j \to \infty$. However, by \eqref{eq:sa}, $D_j \in \mathcal F_{\tau_j} \subset \cup_{i=1}^{\infty} \mathcal F_{\tau_i}$ for all $j \geq 1$. Thus, it follows that $D \subset \sigma(\cup_{i=1}^{\infty} \mathcal F_{\tau_i})$. However, sets of the form $D$ clearly generate $\mathcal F_{\infty}$. The second part of the lemma follows.
\end{proof}

We require the following technical lemma. For this, fix $x_0\in \mathbb R^d$ and let $\mathcal L_{x_0}$ be as in Lemma~\ref{lem:loc}. Suppose that $\{\mathbb P'_{x}\}_{x \in \mathbb R^d}$ is a strong Markov family of solutions to the martingale problem for $\mathcal L_{x_0}$.

\begin{lemma}\label{lem:tech}
Define the measure $\mathbb P'_{x_0} \circ \theta_{\tau_1} (A) = \mathbb P'_{x_0}(\{\omega : \theta_{\tau_1}(\omega) \in A\})$, and let $Q(\omega,\cdot)$ be a regular conditional probability of $\mathbb P'_{x_0} \circ \theta_{\tau_1}$ given $\mathcal F_{\tau_1}$.  For $A \in \mathcal F_{\tau_1}$ and $B \in \mathcal F_{\infty}$, define the event \begin{equation}\label{eq:eab}E_{A,B} = A \cap \{\omega : \theta_{\tau_1}(\omega) \in B\}.\end{equation} Then,
\begin{enumerate}
\item For every $i \geq 0$, $\mathcal F_{\tau_{i+1}} \subset \sigma \{E_{A,B} : A \in \mathcal F_{\tau_1}, B \in \mathcal F_{\tau_i}\}$. Further, $$\mathcal F_{\infty} =\sigma \{E_{A,B} : A \in \mathcal F_{\tau_1}, B \in \mathcal F_{\infty}\}.$$ 
\item For every $x \in \mathbb R^d$, define the measure \begin{equation}\label{eq:q1}
(\mathbb P_x \otimes_{\tau_1} Q)(E_{A,B}) = \mathbb E_{\mathbb P_{x}}[Q(\omega,B)1_{\{\omega \in A\}}]
\end{equation}
for every $A \in \mathcal F_{\tau_1}$ and $B \in \mathcal F_{\infty}$, which can be extended to $\mathcal F_{\infty}$ by the previous part. Then, a right-continuous progressively measurable bounded process $\{M_t\}_{t \geq 0}$ is a $(\mathbb P_x \otimes_{\tau_1} Q)$-martingale if \begin{enumerate}
\item $\{M_{\tau_1 \wedge t}\}_{t \geq 0}$ is a $\mathbb P_{x_0}$- martingale, and
\item For all $\omega \in \Omega$, $\{M_{t} - M_{t \wedge \tau(\omega)}\}_{t\geq 0}$ is a $Q(\omega,\cdot)$-martingale.
\end{enumerate}
\end{enumerate}
\end{lemma}

\begin{proof}
We will first prove (1). By taking $B = \Omega$, the statement is clearly true for $i = 0$. Therefore, we assume that $i > 0$.  

Let $f: \Omega \to \mathbb R$ be $\mathcal F_{\tau_{i+1}}$-measurable. By \eqref{eq:sa} and \cite[Exercise 1.5.6]{SV}, there exists a function $F : (\mathbb R^d)^{\mathbb Z^+} \to \mathbb R$ and an increasing sequence $\{t_n\}_{n \geq 1}$ such that $$
f(\omega) = F(X_{t_1 \wedge \tau_{i+1}}(\omega), X_{t_2 \wedge \tau_{i+1}}(\omega), \ldots).
$$
For $0 \leq s \leq  \tau_1(\omega)$, $s \wedge \tau_{i+1}(\omega)$ can be written as $c \tau_{1}(\omega)$ for some $c \in[0,1]$. On the other hand, for $\tau_1(\omega) \leq s$, $s \wedge \tau_{i+1}(\omega)$ can be written as $c \tau_1(\omega) + (1-c)\tau_{i+1}(\omega)$ for some $c\in[0,1]$. Thus, it follows that $$
\mathcal F_{\tau_{i+1}} = \sigma(\{X_{c\tau_{1}} : c \in [0,1]\} \cup \{X_{c \tau_1 +(1-c)\tau_{i+1}} : c \in [0,1]\}).
$$
Clearly, for every $c \in [0,1]$, $X_{c \tau_1}$ is $\mathcal F_{\tau_1}$-measurable. On the other hand, for any $c \in [0,1]$ and $B \in \mathcal B(\mathbb R^d)$, $$
\{X_{c \tau_1 +(1-c)\tau_{i+1}(\omega)} \in B\} =  E_{\Omega, B'},
$$
where $B' = \{\omega : \theta_{(1-c)\tau_{i}}(\omega) \in B\}$.

Note that $B' \in \mathcal F_{\tau_i}$. Therefore, it follows that $X_{c \tau_1 +(1-c)\tau_{i+1}(\omega)}$ is also measurable with respect to $\sigma \{E_{A,B} : A \in \mathcal F_{\tau_1}, B \in \mathcal F_{\infty}\}$, completing the proof. The second part of (1) follows from the first part and Lemma~\ref{lem:st}. 

Now, we will prove (2). Recall that $Q(\omega,\cdot)$ is the regular conditional probability of $\mathbb P'_{x_0} \circ \theta_{\tau_1}$ given $\mathcal F_{\tau_1}$. For fixed $\omega \in \Omega$, define the measure $\delta_{\omega} \otimes_{\tau_{1}} Q(\omega,\cdot)$ by \begin{equation}\label{eq:delta}
\delta_{\omega} \otimes_{\tau_1} Q(\omega,\cdot)(E_{A,B}) = 1_{\{\omega \in A\}} Q(\omega,B),
\end{equation}
which extends to $\mathcal F_{\infty}$ by part (1).  By \eqref{eq:q1}, for every $A \in \mathcal F_{\tau_1}, B \in \mathcal F_{\infty}$, $$
\mathbb E_{x_0}[\delta_{\omega} \otimes_{\tau_1} Q(\omega,\cdot)(E_{A,B})] = \mathbb P_{x_0} \otimes_{\tau_1} Q (E_{A,B}).
$$
Thus, by part (1) again, it follows that \begin{equation}\label{eq:mre}\mathbb P_{x_0} \otimes_{\tau_1} Q = \mathbb E_{x_0}[(\delta_{\omega} \otimes_{\tau_1} Q(\omega,\cdot)) (\cdot)]\end{equation}

Now let $\{M_t\}_{t \geq 0}$ be any right continuous progressively measurable bounded process. By boundedness, it is also $\mathbb P_{x_0} \otimes_{\tau_1} Q$-integrable. Observe that $\mathcal F_{t}$ is countably generated for each $t \geq 0$ since $\{X_{t}\}_{t \geq 0}$ has right continuous paths. Therefore, \cite[Theorem 1.2.10]{SV} holds in our case.

The proof of part (2) now follows from the same argument as in the proof of \cite[Theorem 6.1.2]{SV}, since \eqref{eq:mre} holds.
\end{proof}

We are now ready to prove Proposition~\ref{prop:loc}.

\begin{proof}[Proof of Proposition~\ref{prop:loc}]

We will use the proof technique of \cite[Theorem 2.2, Section 6]{B}. Recall that $\{\mathbb P_{x}\}_{x \in \mathbb R^d}$ is assumed to be a strong Markov family of solutions to the martingale problem for $\mathcal{L}$. Fix $x_0 \in \mathbb R^d$ and recall the sequence of stopping times $\tau_i$ from \eqref{eq:st}. We will prove that, for each $i \geq 1$, $\mathbb P_{x_0}(A)$ is uniquely determined for each $A \in \mathcal F_{\tau_i}$ by induction. By Lemma~\ref{lem:st}, $\mathcal F_{\infty} = \sigma\left(\cup_{i=1}^{\infty} \mathcal F_{\tau_i}\right)$, therefore $\mathbb P_{x_0}$ is uniquely determined on $\mathcal F_{\infty}$.

We will now work on the base case. Recall that $\{\mathbb P'_{x_0}\}_{x_0 \in \mathbb R^d}$ is a strong Markov family of solutions to the martingale problem for $\mathcal L_{x_0}$, which is uniquely determined by Lemma~\ref{lem:loc}(b) and our assumption. Let $Q^1_{x_0} = \mathbb P_{x_0} \otimes_{\tau_1} Q$ be the measure given in \eqref{eq:q1}. 

We claim that $Q^1_{x_0}$ solves the martingale problem for $\mathcal L_{x_0}$. To do this, let $f \in C_b^2(\mathbb R^d)$ be arbitrary. We must show that $\{M^f_t\}_{t \geq 0}$, given by \eqref{eq:mart}, is a $Q^1_{x_0}$-martingale. By definition, $\{M^f_t\}_{t \geq 0}$ is right-continuous and progressively measurable. Note that for all $x \in \R^d$ and $f \in C_b^2(\R^d)$, we have $$
|f(x+h)-f(x) - 1_{\|h\|<1} (\nabla f(x)\cdot h)| \leq 2 (\|f\|_{\infty} \wedge \|H_f\|_{\infty}).
$$
Thus, by \eqref{gsl} and Assumption~\ref{ass}, for all $x \in \mathbb R^d$,
\begin{align}
\|\mathcal Lf(x)\| \leq 2(\|f\|_{\infty} \wedge \|H_f\|_{\infty})\int_{\R^d \setminus \{0\}} K(x,h)\dd h &\leq 2C(\|f\|_{\infty} \wedge \|H_f\|_{\infty})\int_{\R^d \setminus \{0\}} J(h)\dd h \nonumber\\ & \leq 2CK_0\label{eq:Lbdd2}
\end{align}
Hence, 
$$
|M^f_t|_{\infty} \leq \|f\|_{C_b^2(\mathbb R^d)}(2 + 2CK_0t)
$$
Thus, $M^f_t$ is bounded a.s. for each $t>0$. We will now apply Lemma~\ref{lem:tech}(2) in order to show that it is a $Q^1_{x_0}$-martingale.

Condition (i) of the theorem requires that $\{M^f_{t \wedge \tau_1}\}_{t \geq 0}$ be a $\mathbb P_{x_0}$ martingale. However, observe that for $t>0$ fixed, $$
M^f_{t \wedge \tau_1} = f(X_{t \wedge \tau_1}) - f(x_0) - \int_0^{t \wedge \tau_1} \mathcal L_{x_0} f(X_s) \dd s = f(X_{t\wedge \tau_1}) - f(X_0) - \int_0^{t \wedge \tau_1} \mathcal Lf(X_s) \dd s
$$
by Lemma~\ref{lem:loc}. Thus, $M^f_{t \wedge \tau_1}$ is a $\mathbb P_{x_0}$ martingale by the optional stopping theorem, since $\mathbb P_{x_0}$ is a solution to the martingale problem for $\mathcal{L}$.

Condition (ii) of the theorem requires $\{M^f_t - M^f_{t \wedge \tau_1(\omega)}\}_{t \geq 0}$ to be a $\mathbb P'_{X_{\tau_1}(\omega)}$ martingale for every $\omega \in \Omega$. However, \begin{align*}
M^f_t - M^f_{t \wedge \tau_1(\omega)} &= f(X_t) - f(X_{t \wedge \tau_1(\omega)}) - \int_{t \wedge \tau_1(\omega)}^{t} \mathcal L_{x_0} f(X_s) \dd s
\end{align*}
is a martingale since $\mathbb{P}'_{X_{\tau_1}(\omega)}$ is the solution to the martingale problem for $\mathcal L_{x_0}$ started at $X_{\tau_1(\omega)}$.

Applying Lemma~\ref{lem:tech}(2), it follows that $\{M^f_t\}_{t \geq 0}$ is a martingale. Hence, $Q^1_{x_0}$ solves the martingale problem for $\mathcal L_{x_0}$. By uniqueness, it follows that $Q^1_{x_0} = \mathbb P'_{x_0}$. By \eqref{eq:q1}, for any $A \in \mathcal F_{\tau_1}$ we have $Q^1_{x_0}(A) = \mathbb P_{x_0}(A) = \mathbb P'_{x_0}(A)$. Hence, $\mathbb P_{x_0}(A)$ is uniquely determined on $\mathcal F_{\tau_1}$ by uniqueness of $\mathbb P'_{x_0}$. 

For the induction step, suppose that $\mathbb P_{x_0}$ is uniquely determined on $\mathcal F_{\tau_i}$ for some $i \geq 1$ and every $x_0 \in \mathbb R^d$. By the Strong Markov property, $(\omega,A) \to \mathbb P_{X_{\tau_1}(\omega)}(A)$ is the regular conditional probability distribution of $\mathbb P_{x_0}$ given $\mathcal F_{\tau_1}$.  In particular, if $B \in \mathcal F_{\tau_1}$ and $\theta_{\tau_1}$ denotes the shift operator \eqref{eq:shift}, then 
$$
\mathbb E_{\mathbb P_{x_0}}\left[\{\omega' : \theta_{\tau_1}(\omega') \in B\} | \mathcal F_{\tau_1}\right](\omega) = \mathbb P_{X_{\tau_1}(\omega)}(B).
$$
Thus, it follows that $\mathbb P_{x_0}$ is determined on all sets $A \in \mathcal F_{\tau_1}$, as well as all sets of the form $\{\omega : \theta_{\tau_1}(\omega) \in B\}$, where $B \in \mathcal F_{\tau_i}$. However, this implies that if $$
\Sigma = \sigma\left(E_{A,B} : A \in \mathcal F_{\tau_1}, B \in \mathcal F_{\tau_i}\right),
$$
Then, $\mathbb P_{x_0}$ is determined on $\Sigma$. By Lemma~\ref{lem:tech}(1) we have $\mathcal F_{\tau_{i+1}} \subset \Sigma$. Hence, $\mathbb P_{x_0}$ is determined on $\mathcal F_{\tau_{i+1}}$, completing the proof of the induction step and the proposition.
\end{proof}

\bibliographystyle{plain}
\bibliography{martingale_problem_arXiv}
\end{document}